\documentclass[final]{amsart}

\usepackage{graphicx}
\usepackage{physics}
\usepackage{xfrac}
\usepackage{mathtools}
\usepackage{amssymb}
\usepackage[foot]{amsaddr}
\usepackage{microtype}
\usepackage[hidelinks]{hyperref}
\usepackage[british]{babel}

\usepackage[natbib=true,style=oxyear]{biblatex}
\bibliography{references.bib}

\newcommand{\RR}{{\mathbb{R}}}
\newcommand{\NN}{{\mathbb{N}}}

\newcommand{\CC}{{\mathbb{C}}}
\newcommand{\DD}{{\mathbb{D}}}
\newcommand{\PP}{{\mathbb{P}}}

\newcommand{\opT}{{\mathcal{T}}}
\newcommand{\opA}{{\mathcal{A}}}
\newcommand{\opB}{{\mathcal{B}}}
\newcommand{\opC}{{\mathcal{C}}}
\newcommand{\opE}{{\mathcal{E}}}
\newcommand{\ope}{{\varepsilon}}
\newcommand{\opD}{{\mathcal{D}}}
\newcommand{\opP}{{\mathcal{P}}}
\newcommand{\opL}{{\mathcal{L}}}

\newcommand{\odeE}{\mathcal{E}_{11}}
\newcommand{\odeA}{\tilde{A}}
\newcommand{\odeB}{\tilde{B}}
\newcommand{\odeC}{\tilde{C}}
\newcommand{\odeD}{\tilde{D}}

\DeclareMathOperator{\spn}{span}
\newcommand{\inp}[2]{{\left\langle#1{,}\,#2\right\rangle}}

\theoremstyle{plain}
\newtheorem{theorem}{Theorem}[section]
\newtheorem{lemma}[theorem]{Lemma}
\newtheorem{proposition}[theorem]{Proposition}
\newtheorem{corollary}[theorem]{Corollary}

\theoremstyle{definition}
\newtheorem{definition}{Definition}[section]
\newtheorem{assumption}{Assumption}[section]
\newtheorem{example}{Example}[section]

\theoremstyle{remark}
\newtheorem*{overview}{Overview}
\newtheorem*{acknowledgements}{Acknowledgements}

\title[Computing and Optimizing The $H^2$-norm of DDAEs]{Computing and
  optimizing the $H^2$-norm of delay differential algebraic systems}

\date{25 November 2026 (preprint version)}

\thanks{The codes used to produce the figures and numerical results of this
  article are available at \url{https://doi.org/10.5281/zenodo.24416805}}

\author{Evert Provoost}
\email{evert.provoost@kuleuven.be}

\author{Wim Michiels}
\email{wim.michiels@kuleuven.be}

\address{KU~Leuven, Department of Computer Science, B-3001 Leuven, Belgium}

\keywords{Delay differential algebraic equations, $H^2$-norm, Lanczos tau
  method, robust control, and optimization}

\subjclass[2020]{65L03, 34K06, 34K40, and 34K35}

\begin{document}

\begin{abstract}
	We present a Lanczos tau method for the approximation and optimization of the
  $H^2$-norm of time-delay systems described by delay differential algebraic
  equations of index one. The soundness of this approach is proven under the
  assumption of a finite strong $H^2$-norm. Furthermore, we prove convergence if
  the rational approximation of the exponential underlying the discretization is
  well-behaved and the discretization is stability preserving. Numerical results
  suggest that, for multiple delays, the method converges at cubic rate in the
  discretization degree for systems of retarded type and linearly for those of
  neutral type. In the single delay case, we note geometric convergence of the
  $H^2$-norm for systems of both retarded and neutral type when a symmetric
  basis is chosen. Explicit formulas are derived for the gradient of the
  approximation with respect to system parameters and delays. These allow us to
  compute the entire gradient using only about double the computational time of
  approximating the $H^2$-norm alone. We illustrate how these can be used to
  synthesize robust feedback controllers and stable approximate models. The
  article is concluded by a discussion of how the presented results extend and
  improve for approximations based on splines. We note acceleration of the
  convergence rate by a little over a full order for such a choice. Finally, we
  prove that a Lanczos tau method using a spline based on Legendre orthogonal
  polynomials preserves stability and guarantees convergence of the $H^2$-norm.
\end{abstract}

\maketitle

\section{Introduction}\label{sec:intro}

A useful object to study a linear, time-invariant system's properties is its
so-called transfer function $H : \CC \to \CC^{p \times q}$, representing its
input-output behaviour for zero initial conditions as multiplication in the
Laplace domain, that is
\[
  \hat{\vb{z}}(s) = H(s) \hat{\vb{v}}(s),
\]
where $\hat{\vb{z}}(s) \in \CC^p$ and $\hat{\vb{v}}(s) \in \CC^q$ are the
Laplace transforms of the system's out- and input, respectively. An important
system robustness and performance measure in the field of robust control is its
$H^2$-norm,
\[
	\norm{H}_{H^2} = \sup_{0 < \beta < \infty} \pqty{\frac{1}{2\pi} \int_{-\infty}^\infty \norm{H(\beta + i\omega)}_F^2 \dd{\omega}}^{\frac{1}{2}},
\]
where $i$ is the imaginary unit and $\norm{A}_F = \sqrt{\tr(AA^*)}$ is the
Frobenius norm.

Particularly for the synthesis of controllers, the $H^2$-norm is computationally
more practical than the equally common $H^\infty$-norm, as its square is
differentiable with respect to the system's parameters under mild conditions,
allowing for far more tractable optimization.

The $H^2$-norm is finite and, by the maximum modulus principle, reduces to the
integral along the imaginary axis, i.e.\ $\beta = 0$, when the system's zero
solution is exponentially stable and the system exhibits no feedthrough, in the
sense of the below definition.
\begin{definition}\label{def:feedthrough}
  A system is said to exhibit \emph{feedthrough} from input to output if its
  transfer function $H$ is not strictly proper, i.e.\
  \[
    \lim_{r \to \infty} \sup \Bqty{\norm{H(s)}_F : \Re s \ge 0, \abs{s} \ge r} \neq 0.
  \]
\end{definition}

These properties also make the $H^2$-norm a common choice as error metric in
model order reduction, as the $H^2$-optimal approximation is the exponentially
stable approximate model that minimizes the $L^2$-error along the frequency axis
(see example~\ref{ex:rdde-3}).

\subsection{DDAEs and the strong $H^2$-norm}
In what follows, we are concerned with approximating the $H^2$-norm of systems
described by delay differential algebraic equations (DDAEs) of the form
\begin{equation}
	\label{eq:ddae}
	\begin{aligned}
		E\dot{\vb{x}}(t) & = \textstyle \sum^m_{k=0} A_k \vb{x}(t - \tau_k) + B \vb{v}(t), \\
		\vb{z}(t)        & = C \vb{x}(t),
	\end{aligned}
\end{equation}
where the leading matrix $E$ is allowed to be singular,
$\tau_0 = 0 < \tau_1 < \cdots < \tau_m < \infty$ are the constant delays,
$\vb{x}(t) \in \CC^n$ is the state variable, $\vb{v}(t) \in \CC^q$ is the
input, and $\vb{z}(t) \in \CC^p$ is the output at time~$t$. The system's
transfer function is
\begin{equation}
	\label{eq:tf}
	\textstyle
	G(s) = C\pqty{sE - \sum_{k=0}^m A_k e^{-\tau_k s}}^{-1}B.
\end{equation}

These are particularly suitable for modelling interconnected systems and
networks, where delays are often inherently present and algebraic equations
simplify modelling. Moreover, the introduction of algebraic equations greatly
eases the design of fixed order controllers, and the synthesis of stable
approximate models, as these can now be readily included without needing to
eliminate variables, leaving the parameters exposed for direct optimization. As
we will see, DDAEs also allow---for better or worse---much richer types of
time-delay systems than their apparent simplicity might suggest, such as delays
in the in- and output and even neutral systems, i.e.\ systems with delayed
occurrences of the highest order derivative.

To prevent overly exotic cases, we will assume the system to have at
most Kronecker index one. Whilst there is no full consensus on the definition of
a DDAE's indices, we will, in following with e.g.\ \citet{Fridman2002} and
\citet{Unger2020}, use the following.
\begin{definition}\label{def:idx1}
  Let $\tilde U^*$ and $\tilde V$ be matrices with orthogonal columns such that
  $\spn \tilde U^* = \ker E^*$ and $\spn \tilde V = \ker E$. If
  $\tilde U A_0 \tilde V$ is non-singular, the DDAE~\eqref{eq:ddae} is said to
  be of index one. If $E$ is non-singular, and $\tilde U A_0 \tilde V$ thus a
  $0 \times 0$ matrix, the system is said to be of index zero.
\end{definition}
In other words, we will assume the system with the delayed terms eliminated,
i.e.\ the DAE obtained by setting $A_k = \vb{0}$ for $k > 0$, to have at most
Kronecker index one. An extensive treatise on such systems, including ones of
higher index, is given by \citet{Unger2020}.

As mentioned, even though they are not explicitly present, neutral terms can be
easily added using slack variables. For example,
\[
	\begin{aligned}
		\dot{x}(t) & = \dot{x}(t - 1) + v(t), \\
		z(t)       & = x(t),
	\end{aligned}
	\qq{and}
	\begin{aligned}
		\spmqty{1 & 0                          \\ 1 & 0} \dot{\vb{x}}(t) &= \spmqty{0 & 0 \\ 0 & 1}\vb{x}(t) + \spmqty{0 & 1 \\ 0 & 0}\vb{x}(t - 1) + \spmqty{1 \\ 0}v(t), \\
		z(t)      & = \spmqty{1 & 0}\vb{x}(t),
	\end{aligned}
\]
are equivalent from an input-output and stability perspective. In a similar
fashion, input and output delays can also be added using such slack variables. A
number of additional examples of such reformulations can be found in
appendix~\ref{apx:ddae-forms}.

More treacherously, feedthrough can also hide in model~\eqref{eq:ddae}, as it
can, once again, be achieved using slack variables. As the absence of
feedthrough is a requirement for a finite $H^2$-norm, we will thus need
additional tools to check for its presence. Worse yet, such a hidden feedthrough
can also appear after infinitesimal delay perturbations. The system
\begin{align*}
	\spmqty{0 & 0           & 0 & 0                 \\0&0&0&0\\0&0&0&0\\0&0&0&0}\dot{\vb{x}}(t) & = \begin{aligned}[t]
		\spmqty{\imat{4}}\vb{x}(t)
		            & - \spmqty{0 & 0 & 1 & 0 \\0&0&0&0\\0&0&0&0\\0&0&0&0} \vb{x}(t-\tau_1)
		- \spmqty{0 & 0           & 0 & 0     \\0&0&0&0\\0&0&0&1\\0&0&0&0} \vb{x}(t-\tau_2) \\
		            & - \spmqty{0 & 0 & 0 & 0 \\0&0&0&1\\0&0&0&0\\0&0&0&0} \vb{x}(t-\tau_3)
		- \spmqty{0                           \\0\\0\\1} v(t),\end{aligned}                          \\
	z(t)      & = \spmqty{1 & 1 & 0 & 0} \vb{x}(t),
\end{align*}
for instance, which could quite easily hide within a large model, is equivalent to
\[
	z(t) = v(t - (\tau_1 + \tau_2)) - v(t - \tau_3).
\]
If $\tau_1 + \tau_2 = \tau_3$ this is the zero output system, which has zero as
$H^2$-norm. The slightest perturbation, however, clearly introduces feedthrough,
resulting in an infinite $H^2$-norm.

If the system is neutral, also exponential stability can be sensitive to
infinitesimal perturbations of the delays \citep[see
e.g.][]{HaleVerduynLunel2002}, again yielding an infinite $H^2$-norm after
perturbation. This prompted \citet{Gomez2022} to introduce a strong notion of
the $H^2$-norm that takes such perturbations into account. Let
$\vb*{\tau} = (\tau_1, \dots, \tau_m) \in \RR_+^m$ and
$B_\delta(\vb*{\tau}) = \Bqty{\vb*{\theta} \in \RR_+^m : \norm{\vb*{\tau} - \vb*{\theta}} < \delta}$
the open ball around $\vb*{\tau}$ of radius $\delta$. We can then define the
following norm, which will allow us to provide conditions under which the method
we present in this article is sound.
\begin{definition}\label{def:strng-H2}
	The \emph{strong $H^2$-norm} of the system~\eqref{eq:ddae} is given by
	\[
		\norm{G}^{\mathrm{str}}_{H^2} = \limsup_{\delta \to 0^+}
		\Bqty{\norm{G(\,\cdot\,; \vb*{\theta})}_{H^2} : \vb*{\theta} \in
			B_\delta(\vb*{\tau})}.
	\]
\end{definition}

\subsection{Computing the $H^2$-norm}
It is well known that one can compute the $H^2$-norm for an ordinary delay-free
system by solving a Lyapunov equation \citep[Lemma~4.6]{Zhou1995}.
Generalizations to time-delay systems have previously been proposed. In the
retarded case, where the highest derivative is only ever evaluated in the
present, \citet{Vanbiervliet2011} propose to first approximate the system by one
without delays, e.g.\ by using the pseudospectral collocation approximation of
\citet{Breda2005}, and then use the exact $H^2$-norm of this system as an
approximation of the original. Alternatively, one can solve the so-called delay
Lyapunov equation, which additionally admits neutral systems where, as
mentioned, the highest derivative can now also appear delayed, as presented by
\citet{Jarlebring2011}. Both approaches also allow one to compute the gradient
and thus give methods for $H^2$-optimal synthesis.

Whilst $H^2$-optimal model order reduction can be rephrased as a synthesis
problem, specialized methods have been developed where the targeted
approximation is already a DDAE of index one, however, to the best of our
knowledge and unlike the approach we shall present, these have so far limited
the types of systems under consideration. \Citet{PontesDuff2018} extend the
celebrated iterative rational Krylov algorithm (IRKA) of \citet{Gugercin2008} to
systems with input and output delays. Data-driven approaches, using time-delay
extensions of the Loewner framework, also exist, where, unlike the previous and
our approach, only samples of the transfer function are needed instead of the
entire state space representation. TF-IRKA of \citet{Beattie2012}, the
data-driven analogue of IRKA, is extended to systems whose dynamics only depend
on the state at some fixed delay (and thus not the current state) in
\citet{PontesDuff2015}. Such extensions of the Loewner framework can also be
used directly to find reduced order models of systems with multiple delays, be
it not necessarily $H^2$-optimal, as in \citet{Schulze2018}.

In this work, we extend the approach taken by \citet{Vanbiervliet2011} to
approximate the $H^2$-norm, and the gradient of its square, for DDAEs of index
one. For our extension, we opt to use the Lanczos tau approach of
\citet{ItoTeglas1986} to produce an approximation. This discretization has
recently regained interest; it is extensively used in the work of
\citet{Scholl2024a,Scholl2024b} and showed improved convergence rates for the
$H^2$-norm of time-delay systems of retarded type in
\citet{ProvoostMichiels2024,ProvoostMichiels2025}. Whilst mathematically
equivalent to pseudospectral collocation
\citep[Theorem~4.1]{ProvoostMichiels2024}, the operator formulation it admits
greatly simplifies our argumentation and more naturally suggest the use of
symmetric basis functions, yielding faster convergence in the single-delay case,
as will be discussed in subsection~\ref{ssec:conv}.

\begin{overview}
	In subsection~\ref{ssec:ddaediscr} we extend the discretization based approach
	of \citet{Vanbiervliet2011} to DDAEs of index one, dealing with the additional
	challenges of this more general model class as needed. In particular, we prove
	this method to be sound under the assumption of a finite strong $H^2$-norm. We
	continue in subsection~\ref{ssec:conv} by proving that, under mild
	assumptions, the method converges and illustrate typical convergence rates.
	The required preliminaries for these subsections are discussed in
	section~\ref{sec:prelim}. In section~\ref{sec:control} we turn to the
	synthesis of $H^2$-optimal fixed-order controllers and stable approximate
	models. We do so using gradient based optimization, for which we derive
	analytical expressions of the gradient of the square of the approximation with
	respect to the system matrices and delays in subsection~\ref{ssec:deriv}. We
	illustrate this approach with a number of numerical examples in
	subsection~\ref{ssec:examples}. In section~\ref{sec:spline} we sketch how the
	results of the preceding sections extend, and sometimes simplify, for a
	discretization based on splines. Finally, we summarize our findings and
	suggest future avenues in section~\ref{sec:concl}.
\end{overview}

\section{Preliminaries}\label{sec:prelim}
In the next section, we shall extend an existing method for the approximation of
the $H^2$-norm of systems described by delay differential equations of retarded
type to DDAEs. In this section we review the original method and some earlier
results on the strong $H^2$-norm of DDAEs, which will allow us to guarantee the
soundness of our proposed extension.

\subsection{Finiteness of the strong $H^2$-norm of a DDAE}
In the introduction, we saw that the standard $H^2$-norm is sensitive to
infinitesimal delay perturbations and introduced a strong variant
(Definition~\ref{def:strng-H2}) that takes such perturbations into account.
Previous work allows us to characterize the finiteness of this strong $H^2$-norm
using two components: strong exponential stability and an algebraic test to
detect feedthrough under infinitesimal perturbations. In the remainder of this
subsection we review how to establish both.

Remember that a system is exponentially stable if and only if the spectral
abscissa $\alpha < 0$, where $\alpha$ is the supremum of the real parts of the
characteristic roots of the system. Clearly, $\alpha$ is a function of
$\vb*{\tau}$, we can thus, analogously to before, define the following due to
\citet{HaleVerduynLunel2002}.
\begin{definition}
	The system~\eqref{eq:ddae} is \emph{strongly exponentially stable} if, and
	only if, there exists a finite $\delta > 0$ such that
	\[
		\max_{\vb*{\theta} \in B_\delta(\vb*{\tau})} \alpha(\vb*{\theta}) < 0.
	\]
\end{definition}

Numerically, one can assess strong stability using the method described by
\citet[section~4.3]{Michiels2011}. We can also give an analytical
characterization using the following standard form.

Let $E$ be of rank $n - \nu$. Choose matrices $\bar U^*$ and $\bar V$ with
orthogonal columns, such that $\spn \bar U^* = \ker E^*$ and
$\spn \bar V = \ker E$. Additionally, choose matrices $U^{\perp*}$ and $V^\perp$
with orthogonal columns, such that the columns of $U^{\perp*}$ and $V^\perp$
span the orthogonal complements of $\spn U^*$ and $\spn V$, respectively, and
$\bar U^\perp E \bar V^\perp = I_{n - \nu}$. Such matrices are easily obtained
from, e.g., the singular value decomposition of $E$. Under the assumption that
the system~\eqref{eq:ddae} is of index one, we can define
$\bar A = -{(\bar U A_0 \bar V)}^{-1}$. Applying the transform
$\vb{x}(t) = \bar V^\perp \vb{x}_1(t) + \bar V \vb{x}_2(t)$ and left multiplying
by $\spmqty{\bar U^{\perp} \\ \bar A \bar U}$ gives us, as
$\bar A \bar U A_0 \bar V = -I_\nu$,
\begin{equation}
        \label{eq:stand-form}
        \begin{aligned}
                 & \begin{cases}
                                   \dot{\vb{x}}_1(t) = \sum^m_{k=0} A_{k,11} \vb{x}_1(t - \tau_k) + \sum^m_{k=0} A_{k,12} \vb{x}_2(t - \tau_k) + B_1 \vb{u}(t), \\
                                   \vb{x}_2(t)       = \sum^m_{k=0} A_{k,21} \vb{x}_1(t - \tau_k) + \sum^m_{k=1} A_{k,22} \vb{x}_2(t - \tau_k) + B_2 \vb{u}(t),
                           \end{cases} \\
                 & \quad        \vb{y}(t)      = C_1 \vb{x}_1(t) + C_2 \vb{x}_2(t).
        \end{aligned}
\end{equation}

We now have the following as a direct consequence of Corollary~2.1 of
\citet{HaleVerduynLunel2002}.
\begin{theorem}\label{thm:strng-stab-charac}
	The system~\eqref{eq:stand-form}, and thus also~\eqref{eq:ddae} if of index one, is strongly
	exponentially stable if, and only if, the spectral abscissa is negative for
	the nominal delays, i.e.\ $\alpha(\vb*{\tau}) < 0$, and
	\[
		\max_{\vb*{\theta} \in {[0,2\pi]}^m} \rho\pqty{\textstyle\sum^m_{k=1} A_{k,22} \, e^{i\theta_k}} < 1,
	\]
	where $\rho(A)$ is the spectral radius of $A$.
\end{theorem}

Given strong stability, we still need to verify that no feedthrough exists, nor
can appear under infinitesimal perturbations. To this end, define the following
family of matrices.
\begin{definition}\label{def:matrix-powers}
	Let $\opP_{k_0\dots k_m}$ be the sum of all unique matrix products
	consisting of $k_0$~times $A_{0,22} = -I_\nu$, $k_1$~times $A_{1,22}$, \dots,
	and $k_m$~times $A_{m,22}$.
\end{definition}
For example, $\opP_{0,\dots,0} = I_\nu$, $\opP_{1,2,0,\dots,0} = -A_{1,22}^2$,
and
\[
	\opP_{0,1,2,0,\dots,0} = A_{1,22}A_{2,22}^2 + A_{2,22}A_{1,22}A_{2,22} + A_{2,22}^2A_{1,22}.
\]

These $\opP_{\vb{k}}$ appear naturally in light of results like
Theorem~\ref{thm:strng-stab-charac}. Indeed, from the multinomial theorem
directly follows the following property.
\begin{proposition}\label{prop:sum-power}
	For $r \in \NN$ we have
	\[\textstyle
		\pqty{\sum^m_{k=0} z_k A_{k,22}}^r = \sum_{\substack{\vb{k} \in \NN^{m+1} \\ \norm{\vb{k}}_1 = r}} \binom{r}{\vb{k}}\,\vb{z}^{\vb{k}}\,\binom{\norm{\vb{k}}_1}{\vb{k}}^{-1} \opP_{\vb{k}} = \sum_{\substack{\vb{k} \in \NN^{m+1} \\ \norm{\vb{k}}_1 = r}} \vb{z}^{\vb{k}} \opP_{\vb{k}},
	\]
	where $\vb{z}^{\vb{k}} = z_0^{k_0}\cdots z_m^{k_m}$.
\end{proposition}

Using these $\opP_{\vb{k}}$ we can now formulate the following theorem due to
\citet[Theorem~1]{Gomez2022} as sharpened by \citet[Theorem~13]{Mattenet2022}.
\begin{theorem}\label{thm:strng-h2-charac}
	The system~\eqref{eq:stand-form}, and thus also~\eqref{eq:ddae}, has a finite
	strong $H^2$-norm, i.e.\ $\norm{G}_{H^2}^{\mathrm{str}} < \infty$, if, and only
	if, it is strongly stable and $C_2 \opP_{\vb{k}} B_2 = \vb{0}$ for all
	$\vb{k} \in \NN^{m+1}$. Furthermore, it is sufficient to verify the latter
	condition for $\norm{\vb{k}}_1 < \nu$.
\end{theorem}

\subsection{Computing the $H^2$-norm of an RDDE}\label{ssec:h2rdde}
Before turning to the algebraic case, let us first look at systems described by
purely retarded delay differential equations (RDDEs). We arrive at such a system
by putting $E = I_n$ in~\eqref{eq:ddae}, that is
\begin{equation}
	\label{eq:rdde}
	\begin{aligned}
		\dot{\vb{x}}(t) & = \textstyle \sum^m_{k=0} A_k \vb{x}(t - \tau_k) + B \vb{v}(t), \\
		\vb{z}(t)       & = C \vb{x}(t).
	\end{aligned}
\end{equation}

As noted earlier, we know how to practically compute the exact $H^2$-norm of a
delay-free system, i.e.\ with $A_k = \vb{0}$ if $\tau_k \neq 0$. The method
proposed by \citet{Vanbiervliet2011}, then, is rather straightforward:
approximate the $H^2$-norm of the delayed system by that of a delay-free
approximation. A recent overview of this approach using a Lanczos tau method,
which seems to give improved convergence in some cases compared to the
pseudospectral collocation in Chebyshev extrema originally used, was given by
\citet{ProvoostMichiels2024}, building on the work of \citet{ItoTeglas1986}. We
will detail their approach in the remainder of this section.

The effective state of a time-delay system at time $t$---in the sense of the
minimal required information to uniquely define its dynamics---is given by the
function $\xi_t : [-\tau_m, 0] \to \CC^n$, with
$\xi_t(\theta) = \vb{x}(t + \theta)$. Using this history function $\xi_t$, we
can reformulate~\eqref{eq:rdde} in the following form, essentially rephrasing
the RDDE as an advection PDE \citep{Bernier1978},
\begin{align*}
	\pmqty{\ope_0                \\ I} \dot{\xi}_t &= \pmqty{\sum^m_{k=0} A_k \ope_{-\tau_k} \\ \opD} \xi_t + \pmqty{B \\ \vb{0}} \vb{v}(t), \\
	\vb{z}(t) & = C\ope_0 \xi_t,
\end{align*}
where we, for later convenience, introduce an operator notation with evaluation
functionals $\ope_\theta \xi = \xi(\theta)$, identity $I\xi = \xi$, and
derivative operator $\opD \xi(\theta) = \dv{\theta} \xi(\theta)$.

We can now get to a discretization by approximating $\xi_t$ by
$\xi_{tN} \in \PP^n_N$, where $\PP^n_N$ is the space of polynomials of degree
$N$ mapping from $[-\tau_m, 0]$ to $\CC^n$. As $\opD$ reduces the degree of such
a polynomial by one, the dimensionality of the input side of
$\spmqty{\sum^m_{k=0} A_k \ope_{-\tau_k} \\ \opD}$ is equal to that of the
output side; it maps $\PP_N^n$ to $\CC^n \times \PP_{N-1}^n$. Note, however,
that we do not immediately have the same output space on the left hand side. We
will thus have to replace $I$ by some projection. In a Lanczos tau method this
is achieved by truncating the expansion of each component of $\xi_{tN}$ in a
degree-graded orthogonal polynomial basis $\{\phi_k\}_{k=0}^N$ with support
$[-\tau_m, 0]$ \citep{Lanczos1938}. That is, a set of scalar-valued polynomials
which have as defining property that $\phi_k$ is of degree $k$ and
$\inp{\phi_j}{\phi_k} = 0$ if and only if $j \neq k$, where
$\inp{\,\cdot\,}{\,\cdot\,}$ is the inner product defining the sequence with
induced norm $\norm{\,\cdot\,}$. The componentwise truncation action then
becomes
\[ \textstyle
	(\opT_{N-1} \xi)_j = (\xi)_j - \inp{(\xi)_j}{\phi_N} \frac{\phi_N}{\norm{\phi_N}^2}, \quad j = 1,\dots,n,
\]
which gives us the following state space approximation of the
time-delay system~\eqref{eq:rdde}
\begin{equation}
	\label{eq:rddediscr}
	\begin{aligned}
		\opE'_N \dot{\vb{x}}_N(t) & = \opA_N \vb{x}_N(t) + \opB_N \vb{v}(t), \\
		\vb{z}_N(t)               & = \opC_N \vb{x}_N(t),
	\end{aligned}
\end{equation}
with $\opE'_N = \spmqty{[\ope_0] \\ [\opT_{N-1}]}$,
$\opA_N = \spmqty{\sum^m_{k=0} A_k [\ope_{-\tau_k}] \\ [\opD]}$,
$\opB_N = \spmqty{B \\ \vb{0}}$, and $\opC_N = C[\ope_0]$. Here, $[\,\cdot\,]$
is the appropriate block matrix expression of the operator, or block vector
expression of the functional, and $\vb{x}_N(t)$ is the coefficient vector
representing the polynomial $\xi_{tN}$. Note that, as we think of $\opT_{N-1}$
and $\opD$ as mapping from $\PP_N^n$ to $\PP_{N-1}^n$, the resulting $\opE'_N$
and $\opA_N$ are square matrices.

One such matrix realization can be formed by ordering the coefficients such that
$\xi_{tN} = \sum_{k=0}^N \vb{x}_{N,k}(t) \phi_k$, with
$\vb{x}_{N,k}(t) \in \CC^n$, $k = 0,\dots,N$. Doing so we get as block matrix
realizations of the operators
\[
  {[\ope_\theta]}_k = I_n \phi_k(\theta), \quad {[\opD]}_{jk} = I_n \tfrac{\inp{\phi'_k}{\phi_j}}{\norm{\phi_j}^2}, \qq{and} {[\opT_{N-1}]}_{jk} = I_n \delta_{jk},
\]
where $j = 0,\dots,N-1$ and $k = 0,\dots,N$.

Importantly, the implicit formulation~\eqref{eq:rddediscr} can be recast as a
standard state-space representation. Note that
$\ker \opT_{N-1} = \spn I_n \phi_N$. As a basic property of orthogonal
polynomials, the zeroes of $\phi_N$ lie in the interior of its domain
\citep[Theorem~3.3.1]{Szego1975}, hence $\phi_N(0) \neq 0$. As a consequence,
$\opE'_N$ has $\{\vb{0}\}$ as its kernel and is thus invertible.

As per Lemma~4.6 in \citet{Zhou1995}, we can now approximate
\[
	\norm{G}_{H^2} \approx \norm{G_N}_{H^2} = \sqrt{\tr(\opC_N P \opC_N^T)},
\]
where $G_N$ is the transfer function of~\eqref{eq:rddediscr} and $P$ is the
solution of the Lyapunov equation
$\opA_N P \opE_N^{\prime T} + \opE^\prime_N P \opA_N^T = -\opB_N\opB_N^T$.

\section{Computing the $H^2$-norm of a DDAE}\label{sec:h2ddae}
We now present our proposed approximation method for the $H^2$-norm of
system~\eqref{eq:ddae}. In the first subsection we describe the modifications
needed and prove those to be sound for a system of index one with a finite
strong $H^2$-norm. We continue by proving that the resulting method converges,
under mild assumptions on the approximation, and illustrate typical convergence
rates in the second subsection.

\subsection{Lanczos tau methods for DDAEs}\label{ssec:ddaediscr}
Extending the method of subsection~\ref{ssec:h2rdde} to DDAEs is rather natural.
We start by multiplying by $E$ in the left hand side of the top row of the
discretization~\eqref{eq:rddediscr}, yielding
\begin{equation}
	\label{eq:ddaediscr}
	\begin{aligned}
		\opE_N \dot{\vb{x}}_N(t) & = \opA_N \vb{x}_N(t) + \opB_N \vb{v}(t), \\
		\vb{z}_N(t)              & = \opC_N \vb{x}_N(t),
	\end{aligned}
\end{equation}
with $\opE_N = \spmqty{E[\ope_0] \\ [\opT_{N-1}]}$,
$\opA_N = \spmqty{\sum^m_{k=0} A_k [\ope_{-\tau_k}] \\ [\opD]}$,
$\opB_N = \spmqty{B \\ \vb{0}}$, and $\opC_N = C[\ope_0]$.

The result is a DAE, from which we will now have to eliminate the algebraic part
before we can continue as in the previous section. We start by writing the DAE
such that the differential and algebraic parts are separated, analogous
to~\eqref{eq:stand-form}.

Let $U^*$ and $V$ be matrices with orthogonal columns, such that
$\spn U^* = \ker \opE_N^*$ and $\spn V = \ker \opE_N$. Additionally, let
$U^{\perp*}$ and $V^\perp$ be matrices with orthogonal columns, such that the
columns of $U^{\perp*}$ and $V^\perp$ span the orthogonal complements of
$\spn U^*$ and $\spn V$, respectively. Note that such $U$, $V$, $U^\perp$, and
$V^\perp$ are easily obtained from a rank revealing decomposition of $\opE_N$.
Applying the transform
$\vb{x}_N(t) = V^\perp \vb{x}_{N,1}(t) + V \vb{x}_{N,2}(t)$ and left multiplying
by $\spmqty{U^{\perp} \\ U}$ gives us
\begin{equation}
	\label{eq:proj-ddaediscr}
	\begin{aligned}
		 & \left\{\begin{aligned}
			          \opE_{11} \dot{\vb{x}}_{N,1}(t) & = \opA_{11}\vb{x}_{N,1}(t) + \opA_{12} \vb{x}_{N,2}(t) + \opB_1 \vb{v}(t), \\
			          \vb{0}                          & = \opA_{21}\vb{x}_{N,1}(t) + \opA_{22} \vb{x}_{N,2}(t) + \opB_2 \vb{v}(t),
		          \end{aligned}\right. \\
		 & \quad	\vb{z}_N(t)      = \opC_1 \vb{x}_{N,1}(t) + \opC_2 \vb{x}_{N,2}(t),
	\end{aligned}
\end{equation}
where
\begin{align*}
	\opE_{11} & = U^\perp \opE_N V^\perp, & \opA_{11} & = U^\perp \opA_N V^\perp,          & \opA_{12} & = U^\perp \opA_N V, \\
	\opA_{21} & = U \opA_N V^\perp,      & \opA_{22} & = U \opA_N V,                    & \opB_1    & = U^\perp \opB_N,   \\
	\opB_2    & = U\opB_N,              & \opC_1    & = \opC_N V^\perp, \quad\text{and} & \opC_2    & = \opC_N V.
\end{align*}

Unlike~\eqref{eq:stand-form}, we did not pick $U^\perp$ and $V^\perp$ such that
$\opE_{11}$ is an identity matrix. Whilst this is indeed possible, it can lead
to numerical issues when the largest and smallest non-zero singular values of
$\opE_N$ are highly separated. This also allows us to pick
$\spmqty{U^\perp \\ U}$ and $\spmqty{V^\perp & V}$ to be orthogonal matrices,
giving similar numerical advantages. Nonetheless, $\opE_{11}$ is invertible by
construction, which is sufficient for the remainder of our results.

In our implementation, we use the singular value decomposition to construct
these matrices. Using the algorithm of \citet{Golub1965} this has a
computational cost that is cubic in the number of rows of $\opE_N$, i.e.\
requires $O(n^3N^3)$ operations. This is the same asymptotic cost as solving the
Lyapunov equation with the algorithm of \citet{Hammarling1982}, which we shall
need later in this subsection. An extension of our approach to very large
systems, like is done for RDDEs in \citet{Michiels2019}, would thus need to
address both steps. Note that the discretization does not change the rank
deficiency. It is thus likely possible to first rewrite the system in the
standard form~\eqref{eq:stand-form} and construct~\eqref{eq:proj-ddaediscr}
directly.

To finally eliminate the algebraic part from~\eqref{eq:proj-ddaediscr}, we need
to show $\opA_{22}$ to be invertible. To this end, we first need the following
result, extending Theorem~\ref{thm:strng-stab-charac}.
\begin{lemma}\label{lem:stronger-strn-stab-cond}
	If the system~\eqref{eq:stand-form} is strongly exponentially stable then
	\[
		\max_{\vb{z} \in {\bar\DD}^m} \rho\pqty{\textstyle\sum^m_{k=1} A_{k,22} \, z_k} < 1,
	\]
	where $\bar\DD = \{ z \in \CC : \abs{z} \le 1 \}$ is the closed unit disc.
\end{lemma}
\begin{proof}
	The result is shown through contradiction. Assume~\eqref{eq:stand-form} is
	strongly exponentially stable and there do exist some
	$\abs{z_1}, \dots, \abs{z_m} \le 1$ such that
	\[\textstyle
		\rho\pqty{\sum_{k=1}^m A_{k,22} \, z_k} \ge 1.
	\]
	Then there exists some $\abs{\lambda} \ge 1$ such that
	\[\textstyle
		\det\pqty{a_1 \cdots a_m\lambda I_\nu - \sum_{k=1}^m A_{k,22} \, z_k} = 0,
	\]
	where, for now, $a_1 = \dots = a_m = 1$. As a consequence of Rouché's
	theorem, an eigenvalue cannot disappear unless it moves to infinity. Through
	continuation, and keeping the other variables fixed, we thus have a mapping
	$[1, \infty) \ni a_1 \mapsto z_1(a_1)$. As the spectral radius is bounded by
	any natural matrix norm, $\abs{z_1(a_1)} \to \infty$ as $a_1 \to \infty$,
	there is thus some $a^*_1 \ge 1$ such that $z^*_1 = z_1(a^*_1)$ has modulus
	one. Repeating this procedure, we find
	$\abs*{z^*_1} = \cdots = \abs*{z^*_m} = 1$ and
	$a^*_1, \dots, a^*_m \ge 1$ such that
	\[\textstyle
		\det\pqty{a^*_1 \cdots a^*_m \lambda I_\nu - \sum_{k=1}^m A_{k,22} \, z^*_k} = 0.
	\]
	However, as $\abs{\lambda} \ge 1$, this would mean that
	\[\textstyle
		\rho\pqty{\sum_{k=1}^m A_{k,22} \, z^*_k} \ge \abs*{a^*_1 \cdots a^*_m \lambda} \ge 1,
	\]
	contradicting Theorem~\ref{thm:strng-stab-charac}.
\end{proof}

As we shall see shortly, for our discretization the $z_k$ of the above lemma
correspond to $\phi_N(-\tau_k)$, $k=1,\dots,m$. To use this result then, we
impose the following assumption on the basis.
\begin{assumption}\label{assum:max-in-zero}
	Choose the basis $\{\phi_j\}_{j=0}^N$ such that $\abs*{\phi_N(-\tau_k)} \le \abs*{\phi_N(0)}$
	for $k = 1,\dots,m$.
\end{assumption}
This condition on $\phi_N$ is satisfied for most common choices of orthogonal
polynomials. For many Jacobi polynomials, and in particular those used in our
experiments, this is guaranteed by the well-known limits on their extrema
\citep[see][p.~168]{Szego1975}.

We can then prove the following.
\begin{proposition}\label{prop:A22-inv}
	If the original DDAE~\eqref{eq:ddae} is of index one and strongly exponentially stable, and
	$\phi_N$ satisfies Assumption~\ref{assum:max-in-zero}, then $\opA_{22}$, as
	defined in~\eqref{eq:proj-ddaediscr}, is invertible.
\end{proposition}
\begin{proof}
	Note that the vector $\vb*{\xi} \in \CC^{n(N+1)}$, corresponding to the
	polynomial $\xi \in \PP_N^n$, is in the right nullspace of $\opE_N$ (and thus
	in $\spn V$) if, and only if, $\ope_0 \xi \in \ker E$ and
	$\xi \in \ker \opT_{N-1} = \spn I_n\phi_N$. The tuple
	$(\vb{z}, \vb*{\zeta}) \in \CC^n \times \CC^{nN}$ is in the left nullspace of
	$\opE_N$ (and its complex transpose thus in $\spn U^*$) if, and only if,
	$\vb{z}^* \in \ker{(E[\ope_0])}^* = \ker E^*$ and
	$\vb*{\zeta}^* \in \ker {[\opT_{N-1}]}^* = \{\vb{0}^*\}$. In short
	\begin{gather*}
		\vb*{\xi} \in \spn V \iff \ope_0 \xi \in \ker E \wedge \xi \in \spn I_n\phi_N \quad \text{and}\\
		{(\vb{z}, \vb*{\zeta})}^* \in \spn U^* \iff \vb{z}^* \in \ker E^* \wedge \vb*{\zeta} = \vb{0}.
	\end{gather*}

	These properties imply that there exists a similarity transform $T$, such that
	\[
		\textstyle T \opA_{22} T^{-1} = T U \opA_N V T^{-1} = \spmqty{\bar{U}^* \bar{A}^* \\ \vb{0}}^* \opA_N \, [\phi_N \bar{V}] = \sum^m_{k=0} A_{k,22} \, \phi_N(-\tau_k),\footnote{Loosely
		speaking, $V$ and $U$ project onto the right and left nullspaces of $E$
		respectively, where in $V$ the information can only be `carried' by
		$\phi_N$ and in $U$ the tail has to be the zero function.}
	\]
	where $\bar{U}$, $\bar{A}$, $\bar{V}$, $A_{0,22} = -I_\nu$, $A_{1,22}$, \dots, and
	$A_{m,22}$ are as defined in~\eqref{eq:stand-form}. As $\phi_N(0) \neq 0$
	\citep[Theorem~3.3.1]{Szego1975}, $\frac{1}{\phi_N(0)} \opA_{22}$ is
	invertible by Lemma~\ref{lem:stronger-strn-stab-cond} and
	Assumption~\ref{assum:max-in-zero} and so too $\opA_{22}$.
\end{proof}

Having established the invertibility of $\opA_{22}$, we can apply the Schur
complement formula to eliminate the algebraic part from~\eqref{eq:proj-ddaediscr}
and obtain
\begin{equation}\label{eq:ode-schur}
	\begin{aligned}
		\odeE \dot{\vb{x}}_{N,1}(t) & = \odeA \vb{x}_{N,1}(t) + \odeB \vb{v}(t), \\
		\vb{z}_N(t)                 & = \odeC \vb{x}_{N,1}(t) + \odeD \vb{v}(t),
	\end{aligned}
\end{equation}with
\begin{align*}
	\odeA & = \opA_{11} - \opA_{12}\opA_{22}^{-1}\opA_{21},           & \odeB                           & = \opB_1 - \opA_{12}\opA_{22}^{-1}\opB_2, \\
	\odeC & = \opC_1 - \opC_2\opA_{22}^{-1}\opA_{21}, \quad\text{and}
	      & \odeD                                                     & = - \opC_2\opA_{22}^{-1}\opB_2.
\end{align*}

The remaining challenge novel to DDAEs is then to show that if the
system~\eqref{eq:ddae} does not have feedthrough, we always have $\odeD=\vb{0}$.
Luckily, this turns out to be the case as shown in the following result.
\begin{theorem}\label{thm:no-feedthrough}
	If the DDAE~\eqref{eq:ddae} is of index one, has a finite strong $H^2$-norm, and $\phi_N$
	satisfies Assumption~\ref{assum:max-in-zero}, then the
	system~\eqref{eq:ode-schur} does not have feedthrough, i.e.\
	$\tilde{D} = \vb{0}$.
\end{theorem}
\begin{proof}
	As a finite strong $H^2$-norm implies strong exponential stability, we can
	continue the proof of Proposition~\ref{prop:A22-inv}. Analogously to the
	argument there, we have $T \opB_2 = TU\opB_N = B_2$ and
	$\opC_2 T^{-1} = \opC_N V T^{-1} = C_2$, thus
	\[
		\odeD = -\opC_2 T^{-1} T \opA_{22}^{-1} T^{-1} T \opB_2 = -C_2\pqty{\textstyle\sum^m_{k=0} A_{k,22} \,
		\phi_N(-\tau_k)}^{-1}B_2.
	\]
	Because of the invertibility of $\opA_{22}$ and a standard corollary of the
	Cayley--Hamilton theorem, there exists some polynomial $p$ of finite degree
	such that
	\begin{equation}
		\label{eq:D-with-poly-inv}
		\tilde{D} = -C_2\,p\pqty{\textstyle\sum^m_{k=0} A_{k,22} \, \phi_N(-\tau_k)}\,B_2.
	\end{equation}
	From Proposition~\ref{prop:sum-power} we have, for $r \in \NN$,
	\[\textstyle
		C_2\pqty{\sum^m_{k=0} A_{k,22} \, \phi_N(-\tau_k)}^r B_2 = \sum_{\substack{\vb{k}\in\NN^{m+1} \\ \norm{\vb{k}}_1 = r}} \phi_N(-\tau_0)^{k_0} \cdots \phi_N(-\tau_m)^{k_m} \, C_2\opP_{\vb{k}}B_2.
	\]
	As~\eqref{eq:D-with-poly-inv} entails that $\odeD$ is a weighted sum of such
	terms, we finally have $\odeD = \vb{0}$ from
	Theorem~\ref{thm:strng-h2-charac}.
\end{proof}
As is the case for RDDEs, the system obtained by discretizing a stable system
need not itself be stable. If this were to occur, one can reflect the poles on
the right half-plane across the imaginary axis, as is commonly done in reduced
order modelling \citep[see e.g.][]{Conni2024}. Typically, these poles have a
high frequency. As the transfer function converges rapidly on compact sets (see
Theorem~\ref{thm:GN-conv}), flipping thus induces only a minor additional error
on the transfer function and hence the computed $H^2$-norm. As we shall comment
on more extensively in the next subsection, the used discretization seems to
preserve stability for $N$ sufficiently large, so alternatively one can attempt
to increase $N$ if the discretized system is not stable. In either case,
unstable systems obtained from stable ones are exceedingly rare.

Assuming we have a stable discretization and the conditions for
Theorem~\ref{thm:no-feedthrough} hold, we can finally, as before, approximate
the strong $H^2$-norm as
\begin{equation}
	\label{eq:ddae-h2-approx}
	\norm{G}^{\mathrm{str}}_{H^2} \approx \norm{G_N}_{H^2} = \sqrt{\operatorname{tr}(\odeC P \odeC^T)},
\end{equation}
where $P$ is the solution of
$\odeA P \odeE^T + \odeE P \odeA^T = -\odeB\odeB^T$.

We thus arrive at the following procedure to approximate the strong $H^2$-norm
of the DDAE system~\eqref{eq:ddae}:
\begin{enumerate}
	\item Test for strong exponential stability using the method described in
	      \citet[section~4.3]{Michiels2011}. When this check fails, we
	      have $\norm{G}^{\mathrm{str}}_{H^2} = \infty$.
	\item Test for feedthrough under infinitesimal perturbations using
	      Theorem~\ref{thm:strng-h2-charac}. If this test fails, we again know that
	      $\norm{G}^{\mathrm{str}}_{H^2} = \infty$.
	\item Construct~\eqref{eq:ode-schur} and, if exponentially stable,
	      use~\eqref{eq:ddae-h2-approx} to approximate
	      $\norm{G}^{\mathrm{str}}_{H^2}$. If the discretized system is not stable,
	      first flip the poles on the right half-plane across the imaginary axis.
\end{enumerate}

\subsection{Convergence properties}\label{ssec:conv}
In the previous subsection we proposed to approximate the $H^2$-norm of a DDAE of index one
from a Lanczos tau approximation and addressed the problem of feedthrough in
Theorem~\ref{thm:no-feedthrough}, however, we have yet to show it to converge.
To our knowledge, no proof for the convergence of the method of
\citet{Vanbiervliet2011}, addressing time-delay systems of retarded type,
exists. In this subsection we shall show our approach to converge, assuming the
discretization preserves stability and has a sufficiently good (in the sense of
Assumptions~\ref{assum:conv-exp} through~\ref{assum:A22-bounded}) underlying
approximation of the exponential. As DDAEs of index one form a superset of RDDEs, this gives
an argument for the RDDE case as a side effect. We will conclude the subsection
with a number of experiments illustrating typical convergence rates of the
method. We start with the following result, which is a straightforward extension
of Proposition~3.1 of \citet{ProvoostMichiels2024}.
\begin{proposition}\label{prop:rat-exp}
	The transfer function of \eqref{eq:ddaediscr} can be expressed as
	\[\textstyle
		G_N(s) = C\pqty{sE - \sum_{k=0}^m A_k r_N(s, -\tau_k)}^{-1}B,
	\]
	where $r_N(s, \theta)$ is a rational function of type $(N,N)$ in $s \in \CC$
	and a polynomial of degree $N$ in $\theta \in [-\tau_m, 0]$, such that
	\[
		\begin{cases}
			r_N(s, 0) = 1, \\
			\opD r_N(s, \,\cdot\,) = s\opT_{N-1} r_N(s, \,\cdot\,).
		\end{cases}
	\]
\end{proposition}
\begin{proof}
	Taking the Laplace transform of~\eqref{eq:ddaediscr}, assuming zero initial conditions, gives
	\begin{equation}\label{eq:discr-laplace}
		\begin{cases}
			sE[\ope_0] \hat{\vb{x}}_N(s) = \sum_{k=0}^m A_k [\ope_{-\tau_k}] \hat{\vb{x}}_N(s) + B\hat{\vb{v}}(s), \\
			s[\opT_{N-1}] \hat{\vb{x}}_N(s) = [\opD] \hat{\vb{x}}_N(s),                                            \\
			\hat{\vb{z}}_N(s) = C[\ope_0] \hat{\vb{x}}_N(s),
		\end{cases}
	\end{equation}
	where $\hat{f}(s)$ is the Laplace transform of $f(t)$. Let, for each
  $s \in \CC$, $\vb{r}_N(s)$ be the coefficients of a matrix valued polynomial
  such that $[\ope_\theta] \vb{r}_N(s) = I_n r_N(s, \theta)$. From this
  definition, it is easily verified that this polynomial satisfies
  $[\ope_0]\vb{r}_N(s) = I_n$ and
  $s[\opT_{N-1}] \vb{r}_N(s) = [\opD] \vb{r}_N(s)$. Using these properties we
  see that
	\[\textstyle
		\hat{\vb{x}}_N(s) = \vb{r}_N(s) \pqty{sE[\ope_0]\vb{r}_N(s) - \sum_{k=0}^m A_k [\ope_{-\tau_k}]\vb{r}_N(s)}^{-1} B \hat{\vb{v}}(s)
	\]
	satisfies~\eqref{eq:discr-laplace}, and hence
	\begin{align*}
		G_N(s) & = \textstyle C[\ope_0]\vb{r}_N(s)\pqty{sE[\ope_0]\vb{r}_N(s) - \sum_{k=0}^m A_k [\ope_{-\tau_k}]\vb{r}_N(s)}^{-1} B \\
		       & = \textstyle C\pqty{sE - \sum_{k=0}^m A_k [\ope_{-\tau_k}]\vb{r}_N(s)}^{-1} B \\
		       & = \textstyle C\pqty{sE - \sum_{k=0}^m A_k r_N(s, -\tau_k)}^{-1} B. \qedhere
	\end{align*}
\end{proof}

Later we will need an expression of $G_N$ where the differential and algebraic
parts are separated. Using the standard form of~\eqref{eq:stand-form} we can
rewrite this transfer function as
\[
	G_N(s) = \spmqty{C_1 & C_2}\spmqty{sI_{n - \nu} - A_{N,11}(s) & - A_{N,12}(s) \\ - A_{N,21}(s) & - A_{N,22}(s)}^{-1}\spmqty{B_1 \\ B_2},
\]
where $A_{N,ij}(s) = \textstyle \sum_{k=0}^m A_{k,ij} r_N(s, -\tau_k)$,
$A_{0,22} = -I_\nu$, and all other matrices as in~\eqref{eq:stand-form}.
Using the Schur complement formula we obtain the following expression for $G_N$,
where we elided the $s$ dependence for readability and $A^{-1}$ should be read
as ${A(s)}^{-1}$,
\begin{equation}
	\label{eq:GN-no-E}
	G_N(s) = \spmqty{C_1 & C_2} \spmqty{D^{-1}_N & - D^{-1}_N A_{N,12} A^{-1}_{N,22} \\ - A^{-1}_{N,22} A_{N,21} D^{-1}_N & A^{-1}_{N,22} + A^{-1}_{N,22} A_{N,21} D^{-1}_N A_{N,12} A^{-1}_{N,22}}\spmqty{B_1 \\ B_2},
\end{equation}
with $D_N(s) = sI_{n - \nu} - A_{N,11}(s) - A_{N,12}(s) \, {A_{N,22}(s)}^{-1} A_{N,21}(s)$.

By inspection of the definition of $r_N(s, \theta)$, we see that for every
$s \in \CC$ we have a polynomial approximation of
$[-\tau_m, 0] \ni \theta \mapsto e^{\theta s}$, hence, by comparing $G_N$
to~\eqref{eq:tf}, the discretization~\eqref{eq:ddaediscr} can also be interpreted as
approximating each $s \mapsto e^{-\tau_k s}$ in $G(s)$ by
$s \mapsto r_N(s, -\tau_k)$. The quality of the approximation of $G$ by $G_N$,
and hence the convergence of the $H^2$-norm, thus depends on the quality of the
approximation of the exponential. In what follows we shall need the following
two assumptions on the convergence of these $r_N$.
\begin{assumption}\label{assum:conv-exp}
	The basis $\{\phi_k\}_{k=0}^N$ is chosen such that, for all $s \in \CC$, the
	polynomial $[-\tau_m, 0] \ni \theta \mapsto r_N(s, \theta)$ converges
	uniformly to $\theta \mapsto e^{\theta s}$ as $N \to \infty$.
\end{assumption}
\begin{assumption}\label{assum:rN-bounded}
	The basis $\{\phi_k\}_{k=0}^N$ is chosen such that there is a function $f$ of
	order $o(\abs{\omega}^{\frac{1}{4}})$, i.e.\
	$\lim_{\omega \to \infty} \abs{\omega}^{-\frac{1}{4}} f(\omega) = 0$, bounding
	$\abs*{r_N(i\omega, -\tau_k)}$ for all $k = 0,\dots, m$, $\omega \in \RR$,
	and~$N\in\NN$.
\end{assumption}
These assumptions are suggested by the common experience that, for a good choice
of basis, the result of a tau method (such as, by its definition,
$r_N(s, \,\cdot\,)$) tends to be quasi-optimal, i.e.\ with an error within some factor
$C \ge 1$ from that of the best polynomial approximation of the true solution of
the same degree \citep{Boyd2001}. In the case of Chebyshev polynomials, this
observation is motivated by the spectral accuracy of the derivative
\citep{Tadmor1986}.

Because of the equivalence between a tau method and collocation in the zeroes
of~$\phi_N$ \citep{Lanczos1938}, Assumption~\ref{assum:conv-exp} is guaranteed
by Proposition~5.1 of \citet{Breda2015} if the Lebesgue constant associated with
those zeroes grows sublinearly, i.e.\ is of order $o(N)$.

For Assumption~\ref{assum:rN-bounded} we know of no such readily available
result. As the exponential is an entire function, the best polynomial
approximation of $[-\tau_m, 0] \ni \theta \mapsto e^{s\theta}$ converges
super-geometrically in the supremum norm \citep{Bernstein1912}. On the imaginary
axis, i.e.\ $s = i\omega$, the constant zero polynomial has unit error,
trivially and uniformly bounding the error of the best polynomial approximation
in $\omega$ and $N$. A bounded approximation error to a bounded target function
implies that the best polynomial approximation itself is uniformly bounded in
$\omega$ and~$N$. However, the quasi-optimality factor is problem dependent and
can thus grow with $\omega$. Numerical experiments fortunately suggest that the
growth of the maximal error is only logarithmic in $N$, that it is concentrated
around a single frequency, and that it moves away from the origin at a
superlinear rate, attaining maximal error around
$\omega \sim \frac{1}{\tau_m} N \log(N)$ after which---with exception of those
$\theta$ for which $|r_N(i\omega, \theta)| = 1$ for all $\omega \in \RR$---the
approximation rapidly approaches zero. All this suggests that the obtained
approximations are all well within the requirements of
Assumption~\ref{assum:rN-bounded}.

Using the following lemma we will be able to extend the convergence of
Assumption~\ref{assum:conv-exp} to $G_N$.
\begin{lemma}\label{lem:conv-inv}
  If the pointwise matrix inverse of a matrix valued function $J_N(s)$ is
  bounded in the spectral norm on some compact subset $\Omega$ of the complex
  plane, i.e.\ there is some $M$ such that $\|{J_N(s)}^{-1}\|_2 \le M$ for all
  $s \in \Omega$, and $J_N(s)$ converges as $O(f(N))$ to $J(s)$ for
  $N \to \infty$ in the Frobenius norm on $\Omega$, then ${J_N(s)}^{-1}$
  converges at the same order to ${J(s)}^{-1}$ on~$\Omega$.
\end{lemma}
\begin{proof}
        Let $E_N(s) = J_N(s) - J(s)$, then this result simply follows from
        \begin{align*}
                {J_N(s)}^{-1} & = {(J(s) + E_N(s))}^{-1}                                          \\
                             & = {J(s)}^{-1} - {J(s)}^{-1}E_N(s){J(s)}^{-1} + O(\|E_N(s)\|_F^2) \\
                             & = {J(s)}^{-1} + O(\|{J(s)}^{-1}\|^2_2 \, \norm{E_N(s)}_F)           \\
                             & = {J(s)}^{-1} + O(f(N)).
        \end{align*}
        A derivation for the approximation of the perturbed matrix inverse can be
        found in \citet[section~2.4]{StewartSun1990}. The penultimate step is found
        by vectorizing the middle term.
\end{proof}

We can then prove the following result.
\begin{theorem}\label{thm:GN-conv}
	If Assumption~\ref{assum:conv-exp} holds, then $G_N$ converges uniformly to
	$G$ as $N \to \infty$ on compact sets in $\CC$ that do not contain a
	singularity for $N$ sufficiently large.
\end{theorem}
\begin{proof}
	As a straightforward corollary of Assumption~\ref{assum:conv-exp}, we have, by
	combining the largest error and slowest error decay on a compact set, that the
	function $s \mapsto r_N(s, \theta)$ converges uniformly to
	$s \mapsto e^{\theta s}$ on that set. Since multiplication with a constant
	matrix and addition of functions does not change the convergence properties we
	finally have the required result from Proposition~\ref{prop:rat-exp} and
	Lemma~\ref{lem:conv-inv}.
\end{proof}

Proving convergence of the $H^2$-norm will rely on the local convergence of the
transfer function established in the previous theorem and bounding the `tail'. For
the latter we will need, in conjunction with Assumption~\ref{assum:rN-bounded},
the following.
\begin{assumption}\label{assum:A22-bounded}
	There exists an $M$ such that, for all $N$ and $\omega$, we have the bound
  ${\|A_{N,22}(i\omega)^{-1}\|}_F \le M$.
\end{assumption}
As previously discussed, in experiments the magnitude of the rational
approximations only exceed one near $\omega \sim \frac{1}{\tau_m} N \log(N)$. As
the singular matrices form a set of measure zero, this assumption is thus all
but guaranteed by strong stability and Lemma~\ref{lem:stronger-strn-stab-cond},
and appears to hold in numerical experiments.

This then is sufficient to prove the final convergence result.
\begin{theorem}\label{thm:H2-convergence}
	For a strongly exponentially stable system~\eqref{eq:ddae} of index one, with
  finite strong $H^2$-norm, the method described in
  subsection~\ref{ssec:ddaediscr} converges, i.e.\
  $\norm{G_N}_{H^2} \to \norm{G}_{H^2}$ for $N \to \infty$, when the
  discretization preserves stability for $N$ sufficiently large and
  Assumptions~\ref{assum:conv-exp} through~\ref{assum:A22-bounded} hold.
\end{theorem}
\begin{proof}
	We start by bounding the tails of $G_N(i\omega)$ using the form in
	equation~\eqref{eq:GN-no-E}. Note that, as Assumption~\ref{assum:A22-bounded}
	ensures that $A_{N,22}(i\omega)$ is always invertible, we have
	$C_2 \, {A_{N,22}(i\omega)}^{-1} B_2 = \vb{0}$ through an analogous argument to the
	proof of Theorem~\ref{thm:no-feedthrough}. Additionally,
	Assumptions~\ref{assum:rN-bounded} and~\ref{assum:A22-bounded} imply that
  \[
  {D_N(i\omega)}^{-1} = \tfrac{1}{i\omega}\pqty{I_{n - \nu} - \tfrac{1}{i\omega}A_{N,11}(i\omega) - \tfrac{1}{i\omega}A_{N,12}(i\omega) \, {A_{N,22}(i\omega)}^{-1} A_{N,21}(i\omega)}^{-1}
  \]
	decays at least linearly for $\omega \to \infty$, uniformly in $N$. Using this decay and the same
	assumptions, we see that
	the remaining non-zero terms, and thus $G_N(i\omega)$, decay at least as
	$\abs{\omega}^{-\frac{1}{2}-\alpha}$ for some $\alpha > 0$. As $G(i\omega)$
	decays linearly for $\omega \to \infty$, we thus have that, for all
	$\gamma > 0$, there is some $\omega^* > 0$ such that for all
	$\abs{\omega} \ge \omega^*$ we can bound
	\[
		\max\Bqty{\norm{G(i\omega)}_F, \norm{G_N(i\omega)}_F} \le \gamma \abs{\omega}^{-\frac{1}{2}-\alpha}
	\]
	uniformly in $N$.

	Let $e_N(\omega) = \norm{G(i\omega)}_F^2 - \norm{G_N(i\omega)}_F^2$ and take
	some arbitrarily small $\delta > 0$. The uniform, decaying bound on
	$G(i\omega)$ and $G_N(i\omega)$ then implies that there must be some
	$\omega^{**} > 0$ such that, for all $N$,
	\[\textstyle
		\abs{\frac{1}{2\pi} \int_{-\infty}^{-\omega^{**}} e_N(\omega) \dd{\omega} + \frac{1}{2\pi} \int_{\omega^{**}}^{\infty} e_N(\omega) \dd{\omega}} \le \frac{\delta}{2}.
	\]
	Furthermore, preservation of stability implies that there are no
	singularities on the imaginary axis for $N$ sufficiently large. Together
	with Assumption~\ref{assum:conv-exp}, we can thus apply
	Theorem~\ref{thm:GN-conv} and have $e_N \to 0$ as $N \to \infty$, uniformly
	on $[-\omega^{**}, \omega^{**}]$. Hence, there is some $N^*$ such that, for
	all $N > N^*$,
	\[\textstyle
		\abs{\frac{1}{2\pi} \int_{-\omega^{**}}^{\omega^{**}} e_N(i\omega) \dd{\omega}} \le \frac{\delta}{2}.
	\]
	Taking these two bounds together we have, for all $N \ge N^*$,
	\[\textstyle
		\abs{\frac{1}{2\pi} \int_{-\infty}^\infty e_N(i\omega) \dd{\omega}} \le \delta.
	\]

	Again invoking preservation of stability, there is some $N^{**} \ge N^*$
	such that $\norm{G_N}_{H^2} < \infty$ for all $N \ge N^{**}$. As $\delta$
	was arbitrary and using the positivity of norms we can, for all
	$\epsilon > 0$, choose a $\delta$ such that, for all $N \ge N^{**}$,
	\[\textstyle
		\abs{\norm{G}_{H^2} - \norm{G_N}_{H^2}} \le \epsilon. \qedhere
	\]
\end{proof}
Note that for RDDEs the argument simplifies dramatically as
Assumption~\ref{assum:A22-bounded} becomes irrelevant. We then also have
$D_N(s) = sI_n - A_{N,11}(s)$ and $G_N(s) = C_1 \, {D_N(s)}^{-1} B_1$, thus
Assumption~\ref{assum:rN-bounded} can be relaxed to sublinear growth.

If the discretized system were not to be stable, the proof still implies that the
$L^2$-norm along the imaginary axis converges, as long as the false poles are
not on the imaginary axis. From experience, however, the discretization does
seem to preserve stability, for $N$ sufficiently large, for most classical
choices of basis. This is not unreasonable to assume as similar results on the
preservation of stability have been proven for RDDEs by \citet{Ito1985,
	ItoKappel1992} when using a Legendre basis.

In fact, for this basis we additionally have Assumption~\ref{assum:max-in-zero}
by equation~7.21.1 of \citet{Szego1975} and, as the Lebesgue constants of
Legendre zeroes grow as $O(\sqrt{N})$ \citep[bottom of p.~338]{Szego1975}, we
have Assumption~\ref{assum:conv-exp} through Proposition~5.1 of
\citet{Breda2015}. Hence, up to preservation of stability and
Assumptions~\ref{assum:rN-bounded} and~\ref{assum:A22-bounded}, all conditions
required for Theorem~\ref{thm:H2-convergence} to hold are met when using these
polynomials.

As Legendre polynomials are symmetric on their domain, we have
\[
	\abs{r_N(i\omega, -\tau_m)} = 1
\]
for all $\omega \in \RR$ \citep[Proposition~5.4]{ProvoostMichiels2024}. In the
single delay case we thus trivially satisfy Assumption~\ref{assum:rN-bounded}
and have Assumption~\ref{assum:A22-bounded} from
Lemma~\ref{lem:stronger-strn-stab-cond}, if the system is strongly exponentially
stable. This leaves preservation of stability, which we establish for the single
delay case in the following result, in part inspired by the argument of
\citet{Ito1985}.
\begin{theorem}\label{thm:Legendre-single-delay-stable}
	For a strongly exponentially stable, single delay DDAE system~\eqref{eq:ddae} of index one,
	i.e.\ $m = 1$, the discretization~\eqref{eq:ddaediscr} preserves stability for
	$N$ sufficiently large, when using a Legendre basis.
\end{theorem}
\begin{proof}
	From Theorem~4.2 of \citet{ProvoostMichiels2024} we know that, for this choice
	of basis, $r_N(s, -\tau_m)$ is the $N$th diagonal Padé approximant of
	$s \mapsto e^{-\tau_m s}$. \Citet[Lemma~8]{Birkhoff1965} showed of these
	approximations that they are bounded in magnitude by one on the closed right
	half-plane, i.e.\ $\abs*{r_N(s, -\tau_m)} \le 1$ for all
	$s \in \CC_{\ge 0} = \{z \in \CC : \Re(z) \ge 0\}$.

	Together with Lemma~\ref{lem:stronger-strn-stab-cond}, this bound on $r_N$
	implies that ${A_{N,22}(s)}^{-1}$ is bounded for $s \in \CC_{\ge 0}$. From
	equation~\eqref{eq:GN-no-E} we then conclude that all characteristic roots
	$\lambda$ of the discretization in this half-plane statisfy
	$\det(D_N(\lambda)) = 0$ and thus
	\[
		\abs{\lambda} \le \norm*{A_{N,11}(\lambda) + A_{N,12}(\lambda) \, {A_{N,22}(\lambda)}^{-1} A_{N,21}(\lambda)}_2.
	\]
	The bound on $r_N$ further implies that $A_{N,11}$, $A_{N,12}$, and
	$A_{N,21}$ are bounded on $\CC_{\ge 0}$. Thus, the zeroes of
	$\det(D_N(\lambda))$ that lie in $\CC_{\ge 0}$ are bounded in magnitude by
	some finite constant $M$ and are thus constrained to a hemicircle~$\Omega$
	of radius~$M$.

	As the original system is stable, it has no characteristic roots in
	$\Omega$. From the argument for Theorem~\ref{thm:GN-conv}, we know that
	$D_N(s)$ converges on compact sets. Hence, for $N$ sufficiently large, the
	discretized system cannot have any characteristic roots in $\Omega$ and thus
	has none in the entire closed right half-plane.
\end{proof}

Using this theorem and the preceding discussion we finally have.
\begin{corollary}\label{cor:Legendre-H2-convergence}
	For a strongly exponentially stable, single delay system~\eqref{eq:ddae} of
  index one, with finite strong $H^2$-norm, the method described in
  subsection~\ref{ssec:ddaediscr} converges, i.e.\
  $\norm{G_N}_{H^2} \to \norm{G}_{H^2}$ for $N \to \infty$, when using Legendre
  polynomials as basis.
\end{corollary}

The reason we could provide a complete argument for this case mainly hinges on
the boundedness of $r_N(s, -\tau_m)$. In
\citet[subsection~3.1]{ProvoostMichiels2025} it is shown that, for a spline
based discretization with its knots on the delays, the more general
$\abs{r_N(i\omega, -\tau_k)} = 1$ holds for all delays $\tau_k$ and
$\omega \in \RR$, if the underlying polynomials are symmetric on their domains.
This suggests that this corollary can be extended to multiple delays when using
a spline based discretization, as we shall detail in section~\ref{sec:spline}.

Having established conditions for convergence, we conclude this subsection
discussing the convergence rate. In the case of retarded type systems,
\citet{Vanbiervliet2011} report third order algebraic convergence of
$\norm{G_N}_{H^2}$ to $\norm{G}_{H^2}$ and provide partial arguments for this
observation. Whilst similar derivations are likely possible for our extension,
we limit ourselves to an empirical study of the typical convergence rates seen
for different families of systems that can be modelled as
system~\eqref{eq:ddae}.

For our illustrations we shall use the following group of systems
\begin{equation}\label{eq:conv-sys}
	\begin{aligned}
		\dot{\vb{x}}(t) & = \begin{aligned}[t]
			                    \spmqty{-5 & 1                                                       \\ 3 & 8} \vb{x}(t)
			                               & {}+{} \delta_{r_1}\spmqty{-2            & 0             \\ -2 & 1} \vb{x}(t - 1) -
			                    \delta_{r_2} I_2 \vb{x}(t - 1.9)                                     \\
			                               & {}+{} \delta_{n_1}\spmqty{-\sfrac{1}{5} & \sfrac{1}{10} \\ 0 & -\sfrac{1}{10}} \dot{\vb{x}}(t - 1) -
			                    \tfrac{1}{10} \delta_{n_2} I_2 \dot{\vb{x}}(t - 1.9) +
			                    \spmqty{1                                                            \\ 1} v(t),
		                    \end{aligned}                \\
		z(t)            & = \spmqty{1                                                                                                                                                       & 1} \vb{x}(t),
	\end{aligned}
\end{equation}
where
$\vec{\delta} = (\delta_{r_1}, \delta_{r_2}; \delta_{n_1}, \delta_{n_2}) \in \{0, 1\}^4$
is used to determine the system's composition (see appendix~\ref{apx:ddae-forms}
for the DDAE formulation). As for this choice the single delay case is proven to
converge (see Corollary~\ref{cor:Legendre-H2-convergence}), we use the
appropriately shifted and scaled $N$th Legendre polynomial $P^*_N$ as $\phi_N$
for our discretization. Furthermore, remember that the Legendre polynomials are
symmetric on their domain, which is the property that seems to explain the
geometric convergence observed for the $H^2$-norm of single delay RDDEs in
previous work \citep{ProvoostMichiels2024}.

\begin{figure}
	\centering
	\includegraphics[width=.7\textwidth]{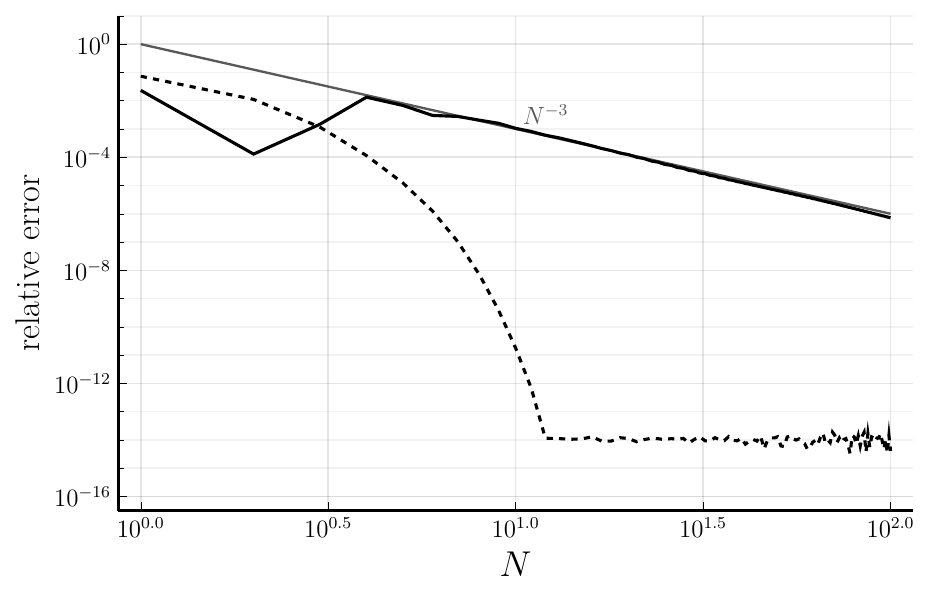}
	\caption{\label{fig:conv-rdde} Convergence rate for different RDDE
		configurations of system \eqref{eq:conv-sys}. The dashed curve corresponds
		to the single delay case, $\vec{\delta} = (1, 0; 0, 0)$, the solid curve
		to two delays, $\vec{\delta} = (1, 1; 0, 0)$.}
\end{figure}

\begin{figure}
	\centering
	\includegraphics[width=.7\textwidth]{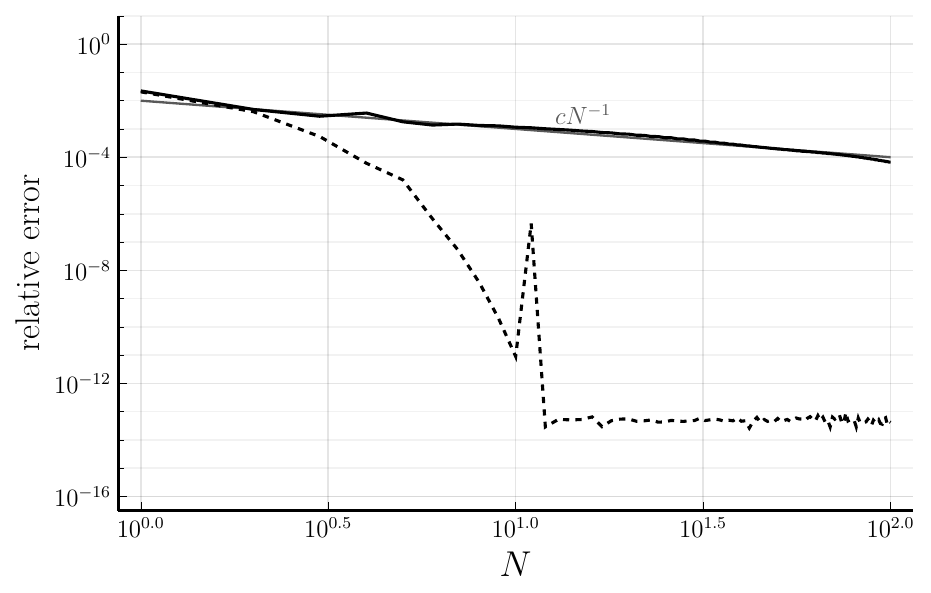}
	\caption{\label{fig:conv-ndde} Convergence rate for different NDDE
		configurations of system \eqref{eq:conv-sys}. The dashed curve has a
		single delay, $\vec{\delta} = (0, 0; 1, 0)$, the solid curve two,
		$\vec{\delta} = (0, 0; 1, 1)$.}
\end{figure}

Hence, for systems of retarded type, we expect third order algebraic convergence
for $m > 1$ and geometric convergence for $m = 1$, which is confirmed on
figure~\ref{fig:conv-rdde}. As systems of neutral type have vertical chains of
infinitely many characteristic roots parallel to the imaginary axis
\citep[chapter~7]{HaleVerduynLunel1993} and the $H^2$-norm involves an integral
along this axis, a reduction in convergence rate is expected for such systems.
Indeed, on figure~\ref{fig:conv-ndde} we see a reduction to first order
algebraic convergence for multiple delays. Surprisingly, in the single delay
case we maintain the geometric convergence seen for RDDEs. As reported in
\citet{ProvoostMichiels2024}, we believe that this is due to the fact that
symmetric $\phi_N$ imply
$\abs*{r_N(i\omega, -\tau_m)} = 1 = \abs*{e^{-\tau_m i\omega}}$ for all
$\omega \in \RR$, something that is not, in general, true at the interior
delays.

\begin{figure}
	\centering
	\includegraphics[width=.7\textwidth]{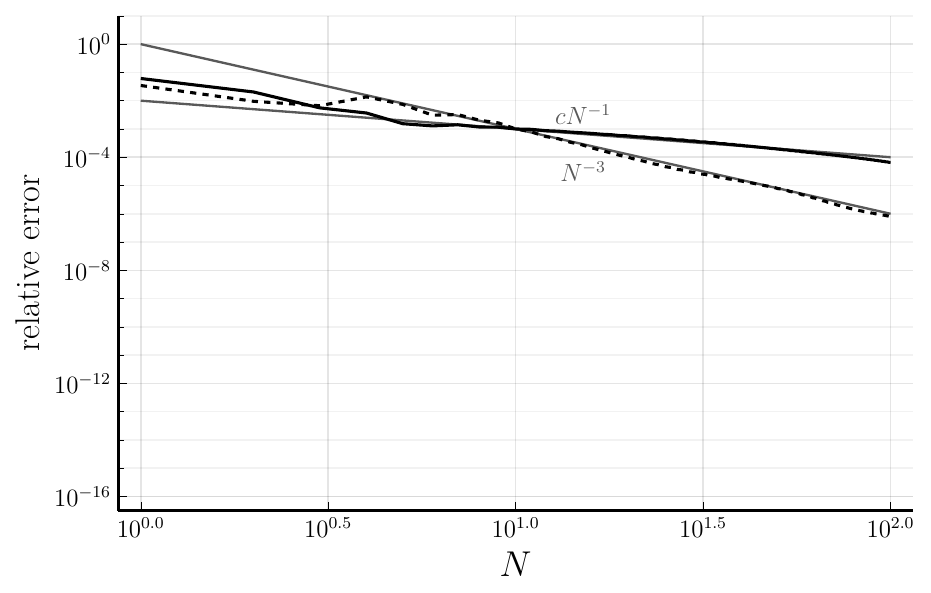}
	\caption{\label{fig:conv-mixed} Convergence rate for different mixed
		configurations of system \eqref{eq:conv-sys}. The dashed curve has a
		neutral term on the maximal delay, $\vec{\delta} = (1, 0; 0, 1)$, the
		solid curve on the interior delay, $\vec{\delta} = (0, 1; 1, 0)$.}
\end{figure}

This reduced effect of the type of the maximal delay $\tau_m$ on the convergence
rate seems to hold even when $m > 1$. On figure~\ref{fig:conv-mixed} we see that
the algebraic convergence rate is only determined by the type of the inner
delays. If any of those correspond to terms that render the equation of neutral
type, we reduce to first order convergence, else we maintain third order
convergence.

\section{$H^2$-optimal synthesis}\label{sec:control}
The goal of our work so far has been an effective method to compute the
$H^2$-norm of a DDAE\@. An additional aim is to use this method to synthesize
$H^2$-optimal fixed-order feedback controllers and stable approximate
models through direct optimization of the model's parameters. In the first
application, system~\eqref{eq:ddae} stems from the feedback interconnection of a
plant and a dynamic or static feedback controller. In the second,
system~\eqref{eq:ddae} describes the error between the dynamics of the full and
the approximate model (see example~\ref{ex:rdde-3}). As a consequence, some
elements of the matrices $A_0$, \dots, $A_m$, $B$, and $C$ may be equal to, or
depend on, the controller or model's parameters.

This motivates the sensitivity analysis in subsection~\ref{ssec:deriv} of the
$H^2$-norm with respect to changes of these matrices, such that we can use a
gradient based optimization method. Likewise, the application to delay based
controllers \citep[see e.g.][]{Villafuerte2013}, where intentional delays are
used to improve the control performance, motivates the study of delay
sensitivity as well.

Finally, in subsection~\ref{ssec:examples} we illustrate our method using a
number of applications, both to robust control as to the synthesis of stable
approximate models.

\subsection{Computing the gradient}\label{ssec:deriv}
In this subsection we discuss the computation of the derivative of the square of
the $H^2$-norm with respect to system parameters. As a norm is always positive,
minimizing its square also minimizes the norm.

The $H^2$-norm of a system of the form~\eqref{eq:ddae} may not be everywhere
differentiable with respect to a parameter $p$. However, if the system matrices
are smooth functions of parameters $p$, we keep $E$ fixed, and we assume the
system to both have differentiation index one (in the sense of
Definition~\ref{def:idx1}) and satisfy the conditions of
Theorem~\ref{thm:strng-stab-charac} in an open set containing $p = p_0$, then
its square is infinitely differentiable at $p = p_0$.

This can be seen as follows. First, the assumed strong stability implies, from a
combination of Propositions~2 and~3, Theorem~1, and Remark~1 of
\citet{Michiels2011}, that there exists an open set $\Omega$ in the combined
parameter and delay space (or, if the parameter is itself a delay, in the delay
space) containing $(p_0,\vb*{\tau})$, such that all characteristic roots lie in
the open left half-plane and are bounded away from the imaginary axis, uniformly
over $\Omega$. Second, the satisfaction of the algebraic conditions of
Theorem~\ref{thm:strng-stab-charac} on an open set around $p = p_0$ implies the
absence of a feedthrough for all $p$ in this set, also if small delay
perturbations are considered. Moreover, in the proof of Theorem~19 of
\citet{Mattenet2022}, an explicit expression for the transfer function is
derived that directly reveals the strictly proper nature under such conditions
(eq.~25 of said article). Finally, these properties imply in turn that, on an
open set containing $p = p_0$, the strong $H^2$-norm is finite, and we thus have
\[
  {(\norm{G}_{H^2}^{\mathrm{str}})}^2 = \norm{G}^2_{H^2} = \frac{1}{2\pi}\int_{-\infty}^{\infty} \norm{G(i\omega; p)}_F^2 \dd{\omega}.
\]
Derivatives with respect to $p$ at $p=p_0$ then exists and can be obtained
directly by differentiating the integral with respect to $p$, relying on the
explicit expression for $G$ mentioned above. In summary, if the conditions of
Theorem~\ref{thm:no-feedthrough} are satisfied in a neighbourhood of the
parameters, the required derivatives are well defined.

Approximating the $H^2$-norm once, using the above approach, takes
$O(n^3N^3)$ elementary operations when using a direct solver for the
Lyapunov equation, such as the one by \citet{Hammarling1982}. Using finite
differences we could then approximate the gradient with respect to the free
parameters, for use in an optimization loop, in $O((r+1)n^3N^3)$ time,
where $r$ is the number of parameters.

An alternative method, with a time complexity independent of $r$, obtains
analytical expressions for the derivatives using an adjoint method, by solving
the dual Lyapunov equation. Hence, we can compute the gradient in a little over
double the computational time of only computing the $H^2$-norm itself. A
disadvantage of this approach is that we compute the derivative of the
discretization's $H^2$-norm and not of the original system. When used in an
optimization loop, some care must also be taken to reject updates that increase
the differentiation index beyond one or blow up the strong $H^2$-norm.

The approach is analogous to the formulas in \citet{Vanbiervliet2009}, however
we opt to use an alternative derivation using matrix differentials, as
introduced by \citet{MagnusNeudecker1985}, which makes the calculus a bit less
tedious. Throughout, we will assume the convention
$\pqty{\dv{f}{X}}_{ij} = \pdv{f}{x_{ij}}$.

Let $\opL P = \odeA P \odeE^T + \odeE P \odeA^T$, then we can write the primary
and dual Lyapunov equations as
\begin{align*}
	\ell_P & \coloneq \opL P + \odeB\odeB^T = \vb{0},\quad\text{and} \\
	\ell_Q & \coloneq \opL^T Q + \odeC^T \odeC = \vb{0},
\end{align*}
respectively.\footnote{Note that indeed
	$\opL^T Q = \odeA^T Q \odeE + \odeE^T Q \odeA$ as
	$\opL^T \cong {(\odeE \otimes \odeA + \odeA \otimes \odeE)}^T$.}

Let $X$ be some variable independent of $P$ and $Q$. Assuming everything but
$P$ and $X$ fixed, we get from the top equation the following linear relation
between infinitesimal perturbations of $P$ and $X$ and the resulting
infinitesimal perturbation of $\ell_P$, written in terms of differentials this
is
\[
	\dd{\ell_P} = \pdv{\ell_P}{P}\dd{P} + \pdv{\ell_P}{X}\dd{X} = \opL\dd{P} + \pdv{\ell_P}{X}\dd{X} = \vb{0},
\]
which implies
\[
	\dd{P} = -\opL^{-1}\pdv{\ell_P}{X} \dd{X}.
\]
We can assume the existence of $\opL^{-1}$ as we only consider this relation
when the solutions $P$ and $Q$ exist.

Rewriting $\ell_Q = \vb{0}$ we get $Q = -\opL^{-T} \odeC^T \odeC$, thus, using
$Q = Q^T$,
\[
	\odeC^T \odeC \dd{P} = Q \pdv{\ell_P}{X} \dd{X}.
\]

Now note that
$\dd{\norm{G_N}^2_{H^2}} = \tr(\odeC \dd{P} \odeC^T) = \tr(\odeC^T \odeC\dd{P})$, hence
\[
	\dd{\norm{G_N}^2_{H^2}} = \tr(Q \pdv{\ell_P}{X} \dd{X}).
\]
Using analogous reasoning, we find from $\norm{G_N}^2_{H^2} = \tr(\odeB^T Q \odeB)$ that
\[
	\dd{\norm{G_N}^2_{H^2}} = \tr(P \pdv{\ell_Q}{X} \dd{X}).
\]
The required derivatives are finally found from simple, but tedious,
computations detailed in appendix~\ref{apx:deriv} and the fact that
$\dd{f} = \tr(Y^T \dd{X})$ implies $\dv{f}{X} = Y$. We finally have
\begin{gather*}
	\dv{\norm{G_N}^2_{H^2}}{A_k} = 2\spmqty{I_n \\ \vb{0}}^T \opD_{\mathcal{S}}^T [\ope_{-\tau_k}]^T, \\
	\dv{\norm{G_N}^2_{H^2}}{\tau_k} = -2\tr(A_k [\ope_{-\tau_k}\opD] \opD_{\mathcal{S}} \spmqty{I_n \\ \vb{0}}) \qq{for} k < m, \\
	\dv{\norm{G_N}^2_{H^2}}{\tau_m} = -2\tr(\tfrac{1}{\tau_m} [\opD] \opD_{\mathcal{S}} \spmqty{\vb{0} \\ I_{nN}}) - \sum_{k =0}^{m-1} \frac{\tau_k}{\tau_m} \dv{\norm{G_N}^2_{H^2}}{\tau_k}, \\
	\dv{\norm{G_N}^2_{H^2}}{B} = 2\spmqty{I_n \\ \vb{0}}^T Q_U^T \odeB, \qq{and}
	\dv{\norm{G_N}^2_{H^2}}{C} = 2\odeC P_V^T [\ope_0]^T,
\end{gather*}
where $Q_U = Q (U^\perp - \opA_{12}\opA_{22}^{-1}U)$,
$P_V = (V^\perp - V\opA_{22}^{-1}\opA_{21}) P$, $W_{\!\opA} = V\opA_{22}^{-1}U$,
and
$\opD_{\mathcal{S}} = P_V\odeE^T Q_U - W_{\!\opA} \opB_N\odeB^TQ_U - P_V \odeC^T
	\opC_N W_{\!\opA}$.

\subsection{Numerical examples}\label{ssec:examples}
We conclude this section by illustrating how the approach we developed so far
can be used to synthesize robust controllers and $H^2$-optimal reduced models.
As optimization strategy we use the Broyden--Fletcher--Goldfarb--Shanno
quasi-Newton method \citep[see e.g.][section~6.1]{NocedalWright2006} with the
line search algorithm by \citet{HagerZang2006}. We compute the gradient using
the formulas of the previous subsection. For a step that would invalidate the
conditions for the existence of the gradient outlined in said subsection, we
take the cost function to be infinite such that the line search algorithm will
naturally guarantee those conditions. For the same reasons as given in
subsection~\ref{ssec:conv}, we choose $\phi_N = P^*_N$, the appropriately scaled
and shifted $N$th Legendre polynomial, for our discretization and take $N = 40$.

For simplicity, we shall introduce the examples in their original form, the
(mostly mechanistic) transformation into a DDAE can be found in
appendix~\ref{apx:ddae-forms}. We first verify our method by reproducing
examples 1--3 from \citet{Gomez2019}, in order.

\begin{example}\label{ex:rdde-1}
	Consider the following closed-loop system from \citet{Vanbiervliet2008}:
	\begin{align*}
		\dot{\vb{x}}(t) & =
		\spmqty{
		-0.08           & -0.03        & 0.2    \\
		0.2             & -0.04        & -0.005 \\
		-0.06           & -0.2         & -0.07
		}
		\vb{x}(t) + \spmqty{-0.1                \\ -0.2 \\ 0.1}\spmqty{p_1&p_2&p_3} \vb{x}(t - 5) + \vb{v}(t), \\
		\vb{z}(t)       & = \vb{x}(t).
	\end{align*}
	Starting from the stabilizing parameters $p_1 = 0.472$, $p_2 = 0.505$, and
	$p_3 = 0.603$, which yield a $H^2$-norm of approximately $8.91$, we find
	$p^*_1 \approx 0.538$, $p^*_2 \approx 0.338$, and $p^*_3 \approx 0.226$
	giving a much lower $H^2$-norm of about $5.70$. All identical to the values
	reported by \citet{Gomez2019}.
\end{example}
\begin{example}\label{ex:rdde-2}
	The next example concerns the design of a proportional-retarded controller for
	a DC servomechanism \citep{Villafuerte2013}:
	\begin{align*}
		\dot{\vb{x}}(t) & = \spmqty{0 & 1            \\-\nu^2-bk_p&-2\delta\nu}\vb{x}(t) + \spmqty{0&0\\bk_r&0} \vb{x}(t-\tau) + \spmqty{0\\b}v(t), \\
		z(t)            & = \spmqty{1 & 0}\vb{x}(t),
	\end{align*}
	where $\nu = 17.6$, $\delta = 0.0128$, and $b = 31$, leaving $\tau$, $k_p$,
	and $k_r$ as control parameters.

	As in \citet{Gomez2019}, optimizing for all three parameters yields no
	finite optimizer as we see $\tau \to 0$, $k_p \to -\infty$, and
	$k_r \to \infty$. Fixing $k_p = 22.57$ as suggested, and starting
	optimization from $\tau = 0.03$ and $k_r = 30$, yields
	$\tau^* \approx 0.0519$ and $k_r^* \approx 17.964$, improving the $H^2$-norm
	from about $0.465$ to $0.223$, values identical to those obtained by
	\citet{Gomez2019}.

	Note that, since we can treat delay optimization directly, the formulation
	of the problem is significantly easier for our method as we did not need a
	change of variables to optimize for $\tau$.
\end{example}
\begin{example}\label{ex:rdde-3}
	Consider the following exponentially stable linearized model of a refinement
	plant by \citet{Ross1971}:
	\begin{align*}
		\dot{\vb{x}}(t) & = A_0\vb{x}(t) + A_1\vb{x}(t-0.1)+B\vb{u}(t), \\
		y(t)            & = C\vb{x}(t),
	\end{align*}
	where
	\begin{gather*}
		A_0 = \spmqty{-4.93&-1.01&0&0\\
			-3.20&-5.30&-12.8&0\\
			6.40&0.347&-32.5&-1.04\\
			0&0.833&11&-3.96},\quad
		A_1 = \spmqty{1.92&0&0&0\\0&1.92&0&0\\0&0&1.87&0\\0&0&0&0.724},\\
		B = \spmqty{1&0\\0&1\\0&0\\0&0},\qq{and}
		C = \spmqty{1&1&1&1}.
	\end{gather*}
	We wish to find a stable approximate model of dimension two,
	\begin{align*}
		\dot{\vb{x}}_r(t) & = A_{0r}\vb{x}_r(t) + A_{1r}\vb{x}_r(t-0.1)+B_r\vb{u}(t), \\
		y_r(t)            & = C_r\vb{x}_r(t),
	\end{align*}
	with minimal $L^2$-error on the transfer function along the imaginary axis.
	We can do so by minimizing the $H^2$-norm of the following system,
	optimizing for $A_{0r}$, $A_{1r}$, $B_r$, and $C_r$,
	\begin{align*}
		\dot{\vb{x}}_e(t) & = \spmqty{A_0 & \vb{0}            \\\vb{0}&A_{0r}} \vb{x}_e(t) + \spmqty{A_1&\vb{0}\\\vb{0}&A_{1r}} \vb{x}_e(t-0.1) + \spmqty{B\\B_r} \vb{u}(t), \\
		y_e(t)            & = \spmqty{C   & -C_r}\vb{x}_e(t).
	\end{align*}

	Starting from the stabilizing matrices
	\[
		A_{0r} = \spmqty{-3&-1\\-3&-2},\quad
		A_{1r} = \spmqty{1&0\\2&0},\quad
		B_r = \spmqty{1.6&0.3\\0.15&0.7},\qq{and}
		C_r = \spmqty{0.7&-0.7},
	\]
	the $H^2$-norm after optimization is approximately $5.91 \cdot 10^{-3}$
	which is near identical to the error on the model obtained in
	\citet{Gomez2019}.
\end{example}

Next, we look at two neutral systems of increasing complexity.

\begin{example}\label{ex:ndde-1}
	We consider the following example, inspired by example~6.2.2 of
	\citet{Byrnes1984}:
	\begin{align*}
		\dot{x}(t) & = -x(t) + x(t - 1) + p_1\dot{x}(t - 1) + p_2x(t - 1) + v(t), \\
		z(t)       & = x(t).
	\end{align*}
	For $p_1 = 0$ and $p_2 = -1$, Byrnes and co-authors showed this system to be
  stable. For this choice of parameters, the $H^2$-norm is approximately $0.71$.
  Using our method, we find an alternative controller with $p^*_1 \approx -0.27$
  and $p^*_2 \approx -1.50$, yielding a $H^2$-norm of about $0.66$.
\end{example}

\begin{example}\label{ex:ndde-2}
	A more interesting system---which also illustrates how one can use slack
	variables in a DDAE formulation to obtain higher order derivatives---stems
	from a linear oscillator under delayed acceleration-derivative feedback with
	additive noise on the measured acceleration and derivative:
	\begin{align*}
		\ddot{x}(t) + 2\xi\dot{x}(t) + x(t) & = p_1(\ddot{x}(t - \tau_1) + v_1(t)) + p_2(\dot{x}(t - \tau_2) + v_2(t)), \\
		z(t)                                & = x(t).
	\end{align*}
	From the stability chart on figure~1 of \citet{Wang2017} we see that for
	$\xi = 0.2$, $p_1 = 0.5$, $p_2 = -20$, $\tau_1=0.2$, and $\tau_2=0.1$ the
	zero solution is exponentially stable, as $\abs{p_1} < 1$ it is also
	strongly stable by Theorem~\ref{thm:strng-stab-charac}. We find a $H^2$-norm of about $3.23$ for this setup.
	Optimizing for $p_2$ we can improve this to about $0.57$ by taking
	$p^*_2 \approx -0.33$. Additionally optimizing for $\tau_1$, lower bounding
	it to $\tau_2 = 0.1$, yields $p^{**}_2 \approx -0.28$ and
	$\tau^{**}_1 \approx 0.1$ with a slightly improved $H^2$-norm of
	approximately $0.53$.
\end{example}

\section{Spline based Lanczos tau methods}\label{sec:spline}
At the end of subsection~\ref{ssec:conv} we proved convergence for the specific
case of a system with a single delay, when discretized using Legendre
polynomials (Corollary~\ref{cor:Legendre-H2-convergence}). The argument mainly
relied on the fact that, for a single delay, the rational approximation of the
exponential underlying this discretization has magnitude one on the imaginary
axis. We noted that results of \citet{ProvoostMichiels2025} imply that such a
property holds at all discrete delays when using a discretization based on a
spline.

In this section we sketch how this property can be used to adapt the theory
developed in the preceding sections to splines. We conclude by illustrating that
the acceleration in convergence rate that motivated the use of splines in the
case of retarded type systems also extends to our more general setting.

Let $\Xi_{\,tN}$ be a set of polynomials
$\{\xi^{(k)}_{tN} : [-\tau_k, -\tau_{k-1}] \to \CC^n \}_{k=1}^m$, each of degree
$N$, i.e.\ a spline with its knots at the delays. Analogous to
subsection~\ref{ssec:ddaediscr}, we can discretize system~\eqref{eq:ddae} by
expanding each $\xi_{tN}^{(k)}$ in an orthogonal, degree-graded basis
$\{\phi^{(k)}_{j} : [-\tau_k, -\tau_{k-1}] \to \CC\}_{j=0}^N$. Note, however,
that this leaves us with an underdetermined system. We complete the discretization
by requiring continuity at the knots, i.e.
$\xi_{tN}^{(k)}(-\tau_k) = \xi_{tN}^{(k+1)}(-\tau_k)$ for $k = 1,\dots,m-1$,
giving us
\begin{equation}\label{eq:spl-discr}
	\begin{aligned}
		\pmqty{E\ope^{(1)}_0                      \\[0.3em] \pqty{\opT^{(k)}_{N-1}}_{\!k=1}^{\!m} \\[0.8em] \pqty{\vb{0}}_{k=1}^{m-1}} \dot\Xi_{\,tN} &= \pmqty{A_0\ope^{(1)}_0 + \sum_{k=1}^m A_k \ope^{(k)}_{-\tau_k} \\[0.6em] \pqty{\opD^{(k)}}_{\!k=1}^{\!m} \\[0.5em] \pqty{\ope^{(k)}_{-\tau_k} - \ope^{(k+1)}_{-\tau_k}}_{\!k=1}^{\!m-1}} \Xi_{\,tN} + \pmqty{B \\ \vb{0} \\ \vb{0}} \vb{v}(t), \\[1em]
		\vb{z}_N(t) & = C\ope^{(1)}_0 \Xi_{\,tN}, \\[0.5em]
	\end{aligned}
\end{equation}
where $\opP^{(k)}$ indicates that $\opP$ is applied to segment $\xi^{(k)}_{tN}$
and $\pqty{\opP_k}_{k = 1}^m$ is the vertical concatenation of
$\{\opP_1, \dots, \opP_m\}$.

Following \citet{ProvoostMichiels2025}, we can readily extend
Proposition~\ref{prop:rat-exp} to find that the transfer function
of~\eqref{eq:spl-discr} is given by
\[\textstyle
	G^{\rm spl}_N(s) = C\pqty{sE - A_0 - \sum_{k=1}^m A_k r^{(k)}_N(s, -\tau_k)}^{-1}B,
\]
where $r^{(k)}_N(s, \theta)$ is a rational function of $s$ and polynomial in
$\theta \in [-\tau_k, -\tau_{k-1}]$, such that
\[
	\begin{cases}
		r_N^{(1)}(s, 0) = 1,                                                                   \\
		\opD r_N^{(k)}(s, \,\cdot\,) = s\opT_{N-1} r_N^{(k)}(s, \,\cdot\,), & k = 1, \dots, m, \\
		r_N^{(k)}(s, -\tau_{k-1}) = r_N^{(k - 1)}(s, -\tau_{k-1}),          & k = 2, \dots, m.
	\end{cases}
\]

As shown in subsection~3.1 of the same article, we have
\begin{equation}
	\label{eq:spl-mag-one}
	\abs*{r^{(k)}_N(i\omega, -\tau_k)} = 1, \quad k=1,\dots,m,
\end{equation}
for all $\omega \in \RR$, if the underlying polynomials are symmetric on their
domains. Furthermore, each $r_N^{(k)}(s, -\tau_k)$ has the preceding rational
approximations, i.e.\ those with lower $k$, as factors. More specifically, we
have $r_N^{(1)} = \rho^{(1)}_N$ and
\begin{equation}
	\label{eq:spl-rN-recur}
	r_N^{(k)}(s, \theta) = r_N^{(k-1)}(s, -\tau_{k-1}) \rho^{(k)}_N\pqty{s, \theta}, \quad k=2,\dots,m,
\end{equation}
where $\rho^{(k)}_N(s, \theta)$ is a rational function of $s$ and polynomial in
$\theta \in [-\tau_k, -\tau_{k-1}]$, such that
\[
	\begin{cases}
		\rho^{(k)}_N(s, -\tau_{k-1}) = 1, \\
		\opD \rho^{(k)}_N(s, \,\cdot\,) = s\opT_{N-1} \rho^{(k)}_N(s, \,\cdot\,).
	\end{cases}
\]

These two properties simplify the arguments needed for the polynomial case in
section~\ref{sec:h2ddae} somewhat. First, note that the absence of feedthrough
under the condition of strong stability becomes rather straightforward
using~\eqref{eq:GN-no-E}, which, \emph{mutatis mutandis}, applies to
$G_N^{\rm spl}$ as well. Using~\eqref{eq:spl-mag-one} and
Theorem~\ref{thm:strng-stab-charac}, it is easily seen that $A_{N,22}(i \omega)$
is invertible for all $\omega \in \RR$. As ${D_N(i \omega)}^{-1} \to 0$ as
$\omega \to \infty$, the feedthrough term is given by
\[
	\lim_{\omega \to \infty} G_N^{\rm spl}(i \omega) = \lim_{\omega \to \infty} C_2 \,{A_{N,22}(i\omega)}^{-1} B_2.
\]
Using the argument of Theorem~\ref{thm:no-feedthrough}, we readily have
$C_2 \, {A_{N,22}(i\omega)}^{-1} B_2 = \vb{0}$ for all $\omega$ and so too at infinity.

The theorems on the convergence of our $H^2$-norm approximation from
subsection~\ref{ssec:conv} also readily extend to splines. Better yet, most of
the required assumptions now trivially hold. The bound of
Assumption~\ref{assum:rN-bounded} is only required to hold in the delays, which
is trivially true when the underlying polynomials are symmetric, due
to~\eqref{eq:spl-mag-one}. From the same property and
Theorem~\ref{thm:strng-stab-charac}, Assumption~\ref{assum:A22-bounded} holds.

All that is left to show is thus convergence of the $r^{(k)}_N$, i.e.\
Assumption~\ref{assum:conv-exp}, and preservation of stability. As in
subsection~\ref{ssec:conv}, we can do so when we use splines based on Legendre
polynomials. Using the argument for the convergence of $r_N$, we can show that
$\rho^{(k)}_N(s, \theta) \to e^{(\theta + \tau_{k-1}) s}$ as $N \to \infty$.
Convergence of $r^{(k)}_N$ then follows from~\eqref{eq:spl-rN-recur} and
$e^a e^b = e^{a + b}$. For preservation of stability, note that for this choice
of basis, equation~\eqref{eq:spl-rN-recur} implies that
$r^{(k)}_N(i\omega, -\tau_k)$ is the product of $k$ diagonal Padé approximants
of various decaying exponential functions. The argument and conclusion of
Theorem~\ref{thm:Legendre-single-delay-stable} thus readily extend to this case,
giving the following result.
\begin{theorem}\label{thm:spl-Legendre-stable}
	For a strongly exponentially stable DDAE system~\eqref{eq:ddae} of index one, the
	discretization~\eqref{eq:spl-discr} preserves stability for $N$ sufficiently
	large, when using a spline based on a Legendre basis.
\end{theorem}

From the preceding discussion we thus have the following generalization of
Corollary~\ref{cor:Legendre-H2-convergence}.
\begin{corollary}\label{cor:spline-Legendre-H2-convergence}
	For a strongly exponentially stable DDAE system~\eqref{eq:ddae} of index one,
  with finite strong $H^2$-norm, the approach of
  subsection~\ref{ssec:ddaediscr}, modified as described above, converges, i.e.\
  $\norm*{G^{\rm spl}_N}_{H^2} \to \norm{G}_{H^2}$ for $N \to \infty$, when
  using discretization~\eqref{eq:spl-discr} with a spline based on a Legendre
  basis.
\end{corollary}

\begin{figure}
	\centering
	\includegraphics[width=.7\textwidth]{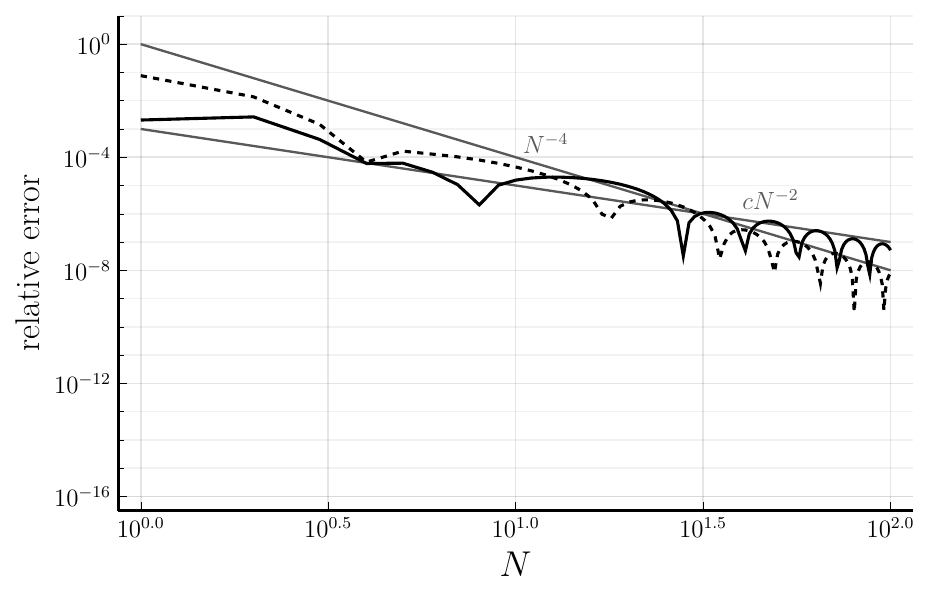}
	\caption{\label{fig:spline-conv-mixed} The same experiment as on
		figure~\ref{fig:conv-mixed} but using a spline discretization. The dashed
		curve has a neutral term on the maximal delay, the solid curve on the
		interior delay.}
\end{figure}

Finally, \citet{ProvoostMichiels2025} demonstrated that the convergence rate for
a RDDE with multiple delays can be accelerated from third order algebraic
convergence to a little over fourth order when using these types of approximations. As
geometric convergence in the single delay case, which is presumed to allow the
acceleration, appears to be maintained in the presented extension to DDAEs (see
subsection~\ref{ssec:conv}), we would expect similar improvements here. Indeed,
comparing figure~\ref{fig:spline-conv-mixed} to figure~\ref{fig:conv-mixed} we
see that a similar improvement of a little over a full order seems to hold for DDAEs. For
the worst case, NDDEs with neutral terms on internal delays, this results in an
improvement from first order to about second order convergence.

\section{Conclusions}\label{sec:concl}
We presented a method to approximate the $H^2$-norm of a DDAE
system~\eqref{eq:ddae} of index one (in the sense of Definition~\ref{def:idx1}),
using a Lanczos tau discretization as proxy. We showed that, under the
assumption of a finite strong $H^2$-norm, this approach is sound
(Theorem~\ref{thm:no-feedthrough}). Furthermore, we were able to prove
convergence under reasonable assumptions, which, to our knowledge, is novel even
in the more restricted case of a system of retarded type
(Theorem~\ref{thm:H2-convergence}). For the specific case of a Legendre based
discretization of a single delay system, we were able to formally prove
preservation of stability (Theorem~\ref{thm:Legendre-single-delay-stable}) and
the other requirements (Corollary~\ref{cor:Legendre-H2-convergence}). Underlying
the convergence of the $H^2$-norm is the convergence of the transfer function
established in Theorem~\ref{thm:GN-conv}.

In section~\ref{sec:control} we derived analytical expressions for the gradient
of the square of the $H^2$-norm of the discretized system, with respect to the
system matrices and the delays. These expressions, together with the $H^2$-norm,
can be evaluated for slightly more than double the computational cost of
computing the $H^2$-norm alone. This allows us to efficiently synthesize both
$H^2$-optimal fixed-order controllers and stable approximate models. We
illustrated this in subsection~\ref{ssec:examples} by reproducing three known
examples for retarded type systems and two novel examples involving neutral
terms.

Finally, we discussed the use of a spline instead of a single polynomial for the
discretization. We saw that this tends to simplify the theory and, when using a
spline based on a Legendre basis, allows us to prove preservation of stability
(Theorem~\ref{thm:spl-Legendre-stable}) and convergence
(Corollary~\ref{cor:spline-Legendre-H2-convergence}) for DDAEs with multiple
delays.

Although we did not treat it explicitly here---as we do not know of any
analogues to Theorems~\ref{thm:strng-stab-charac}
and~\ref{thm:strng-h2-charac}---the presented method can be straightforwardly
extended to more general semi-explicit functional differential algebraic
equations, by replacing $\sum^m_{k=0} A_k [\ope_{-\tau_k}]$ by the discretized
functional in~\eqref{eq:ddaediscr}.

Whilst tractable for moderately sized systems, the computational cost of our
approach scales cubically in both the system size and the discretization degree.
An extension to large scale systems using projection methods, as is done for
RDDEs in \citet{Michiels2019}, might be possible, but would need to address the
decomposition of $\opE_N$, as discussed in subsection~\ref{ssec:ddaediscr}.

\begin{acknowledgements}
	The authors wish to thank Nick Dewaele for providing the proof of
  Lemma~\ref{lem:conv-inv}. This work was supported by KU~Leuven Internal Funds
  project C14/22/092.
\end{acknowledgements}

\appendix
\section{Intermediate steps to compute the derivatives}\label{apx:deriv}
To ease legibility, let us write $\gamma^2 = \norm{G_N}^2_{H^2}$. Starting from
the relations
\[\textstyle
	\dd{\gamma^2} = \tr(Q \pdv{\ell_P}{X} \dd{X}) \qq{and}
	\dd{\gamma^2} = \tr(P \pdv{\ell_Q}{X} \dd{X}),
\]
obtained in subsection~\ref{ssec:deriv}, we straightforwardly get the expression
\begin{align*}
	\dd{\gamma^2} & = \tr(Q (\dd{\odeA} P \odeE^T + \odeE P \dd{\odeA^T})) + \tr(Q (\dd{\odeB} \odeB^T + \odeB\dd{\odeB^T})) \\
	              & \quad + \tr(P (\dd{\odeC^T} \odeC + \odeC^T\dd{\odeC}))                                                  \\
	              & = 2\tr(P\odeE^T Q \dd{\odeA} + Q \dd{\odeB} \odeB^T + P \odeC^T \dd{\odeC}).
\end{align*}

We continue by writing the differentials in terms of the original system,
effectively retracing our steps of subsection~\ref{ssec:ddaediscr}. From the
application of the Schur complement formula we have
\begin{align*}
	\dd{\odeA} & = \dd{\opA_{11}} - \dd{\opA_{12}}\opA_{22}^{-1}\opA_{21} + \opA_{12}\opA_{22}^{-1}\dd{\opA_{22}}\opA_{22}^{-1}\opA_{21} - \opA_{12}\opA_{22}^{-1}\dd{\opA_{21}},    \\
	\dd{\odeB} & = \dd{\opB_1} - \dd{\opA_{12}}\opA_{22}^{-1}\opB_2 + \opA_{12}\opA_{22}^{-1}\dd{\opA_{22}}\opA_{22}^{-1}\opB_2 - \opA_{12}\opA_{22}^{-1}\dd{\opB_2},\quad\text{and} \\
	\dd{\odeC} & = \dd{\opC_1} - \dd{\opC_2}\opA_{22}^{-1}\opA_{21} + \opC_2\opA_{22}^{-1}\dd{\opA_{22}}\opA_{22}^{-1}\opA_{21} - \opC_2\opA_{22}^{-1}\dd{\opA_{21}}.
\end{align*}
Differentiating through the construction of the standard form gives
\begin{gather*}
	\spmqty{\dd{\opA_{11}} & \dd{\opA_{12}} \\ \dd{\opA_{21}} & \dd{\opA_{22}}} = \spmqty{U^\perp \\ U} \dd{\opA_N} \spmqty{V^\perp & V},\quad
	\spmqty{\dd{\opB_1} \\ \dd{\opB_2}} = \spmqty{U^\perp \\ U} \dd{\opB_N},\quad\text{and}\\
	\spmqty{\dd{\opC_1} & \dd{\opC_2}} = \dd{\opC_N} \spmqty{V^\perp & V}.
\end{gather*}
Finally, the discretization~\eqref{eq:ddaediscr} brings us back to the system's
parameters
\begin{gather*}
	\dd{\opA_N} =  \spmqty{I_n    \\ \vb{0}} \textstyle\sum^m_{k=0} \pqty{\dd{A_k} [\ope_{-\tau_k}] + A_k \dd{[\ope_{-\tau_k}]}} + \spmqty{\vb{0} \\ I_{nN}} \dd{[\opD]},\\
	\dd{\opB_N} = \spmqty{I_n    \\ \vb{0}} \dd{B}, \qq{and} \dd{\opC_N} = \dd{C} [\ope_0].
\end{gather*}

Letting $\psi_k(\theta) = \phi_k(\tau_m\theta)$, we can rewrite the block matrix
realization of $\ope_\theta$ of subsection~\ref{ssec:h2rdde} as
${[\ope_\theta]}_k = I_n \psi_k\pqty{\tfrac{\theta}{\tau_m}}$. Using
${[\ope_\theta \opD]}_k = I_n \phi'_k(\theta) = I_n \frac{1}{\tau_m}\psi'_k\pqty{\frac{\theta}{\tau_m}}$
we find
$\dd{[\ope_{-\tau_k}]} = [\ope_{-\tau_k}\opD]\pqty{\tfrac{\tau_k}{\tau_m}\dd{\tau_m} - \dd{\tau_k}}$.
Through similar reasoning we have
$\dd{[\opD]} = -\frac{1}{\tau_m}[\opD] \dd{\tau_m}$.

Using the above, we can then derive $\dd{\gamma^2}$ as a function of $\dd{A_k}$,
$\dd{\tau_k}$ for $k < m$, $\dd{\tau_m}$, $\dd{B}$, and $\dd{C}$, respectively
\begingroup\allowdisplaybreaks
\begin{align*}
	\dd{\gamma^2} ={}  & \begin{aligned}[t] 2\operatorname{tr}\Big([\ope_{-\tau_k}] \big( & (V^\perp - V\opA_{22}^{-1}\opA_{21})P\odeE^T Q(U^\perp - \opA_{12}\opA_{22}^{-1}U)    \\*
                                                                 & - V\opA_{22}^{-1}\opB_2\odeB^TQ(U^\perp - \opA_{12}\opA_{22}^{-1}U)                   \\*
                                                                 & - (V^\perp - V\opA_{22}^{-1}\opA_{21})P\odeC^T \opC_2\opA_{22}^{-1}U\big) \spmqty{I_n \\ \vb{0}} \dd{A_k}\Big),
	                     \end{aligned}                                      \\
	\dd{\gamma^2} ={}  & \begin{aligned}[t] -2\operatorname{tr}\Big(\big( & (V^\perp - V\opA_{22}^{-1}\opA_{21})P\odeE^T Q(U^\perp - \opA_{12}\opA_{22}^{-1}U)    \\*
                                                 & - V\opA_{22}^{-1}\opB_2\odeB^TQ(U^\perp - \opA_{12}\opA_{22}^{-1}U)                   \\*
                                                 & - (V^\perp - V\opA_{22}^{-1}\opA_{21})P\odeC^T \opC_2\opA_{22}^{-1}U\big) \spmqty{I_n \\ \vb{0}} A_k [\ope_{-\tau_k}\opD]\Big) \dd{\tau_k},\quad k < m,
	                     \end{aligned}              \\
	\dd{\gamma^2} = {} & \begin{aligned}[t] -2\operatorname{tr}\Big(\big( & (V^\perp - V\opA_{22}^{-1}\opA_{21})P\odeE^T Q(U^\perp - \opA_{12}\opA_{22}^{-1}U)       \\*
                                                 & - V\opA_{22}^{-1}\opB_2\odeB^TQ(U^\perp - \opA_{12}\opA_{22}^{-1}U)                      \\*
                                                 & - (V^\perp - V\opA_{22}^{-1}\opA_{21})P\odeC^T \opC_2\opA_{22}^{-1}U\big) \spmqty{\vb{0} \\ I_{nN}} \tfrac{1}{\tau_m}[\opD]\Big) \dd{\tau_m}
	                     \end{aligned} \\*
	                   & \quad- \sum_{k =0}^{m-1} \frac{\tau_k}{\tau_m} \dv{\gamma^2}{\tau_k} \dd{\tau_m},                                                                                                                                        \\
	\dd{\gamma^2} ={}  & 2\tr(\odeB^TQ (U^\perp - \opA_{12}\opA_{22}^{-1}U) \spmqty{I_n                                                                                                                                                           \\ \vb{0}}\dd{B}),\quad\text{and}                                                \\
	\dd{\gamma^2} ={}  & 2\tr([\ope_0](V^\perp - V\opA_{22}^{-1}\opA_{21})P\odeC^T\dd{C}).
\end{align*}
\endgroup
Using the abbreviations $Q_U = Q (U^\perp - \opA_{12}\opA_{22}^{-1}U)$,
$P_V = (V^\perp - V\opA_{22}^{-1}\opA_{21}) P$,
$W_{\!\opA} = V\opA_{22}^{-1}U$, and $\opD_{\mathcal{S}} = P_V\odeE^T Q_U - W_{\!\opA} \opB_N\odeB^TQ_U
	- P_V \odeC^T \opC_N W_{\!\opA}$ we get the compacter
\begin{gather*}
	\dd{\gamma^2}  = 2\tr([\ope_{-\tau_k}] \opD_{\mathcal{S}} \spmqty{I_n                                                   \\ \vb{0}} \dd{A_k}),                                                               \\
	\dd{\gamma^2}  = -2\tr(\opD_{\mathcal{S}} \spmqty{I_n                                                   \\ \vb{0}} A_k [\ope_{-\tau_k}\opD]) \dd{\tau_k}, \quad k < m,                                                               \\
	\dd{\gamma^2}  = -2\tr(\opD_{\mathcal{S}} \spmqty{\vb{0} \\ I_{nN}} \tfrac{1}{\tau_m} [\opD])\dd{\tau_m} - \sum_{k =0}^{m-1} \frac{\tau_k}{\tau_m} \dv{\gamma^2}{\tau_k}\dd{\tau_m},                                                               \\
	\dd{\gamma^2}  = 2\tr(\odeB^TQ_U\spmqty{I_n                                                     \\ \vb{0}}\dd{B}),\qq{and}
	\dd{\gamma^2}  = 2\tr([\ope_0]P_V\odeC^T\dd{C}).
\end{gather*}
From the fact that $\dd{f} = \tr(Y^T \dd{X})$ implies $\dv{f}{X} = Y$, we find
the final results of subsection~\ref{ssec:deriv}.

\section{Full system matrices of the numerical examples}\label{apx:ddae-forms}
We detail the DDAE formulations of the examples in subsections~\ref{ssec:conv}
and~\ref{ssec:examples} as passed to the $H^2$-norm optimization routine. This
does not include example~\ref{ex:rdde-3} as there the system can be used as
written in the main text.

\emph{Subsection~\ref{ssec:conv}.} Using slack variables $x_3(t) = \dot{x}_1(t)$
and $x_4(t) = \dot{x}_2(t)$, we get
\begin{align*}
	\spmqty{I_2 & \vb{0}                              \\ I_2 & \vb{0}} \dot{\vb{x}}(t) & = \begin{aligned}[t]
		\spmqty{\spmqty{-5 & 1                                                           \\ 3 & 8} & \vb{0} \\ \vb{0} & I_2} \vb{x}(t) &{}+{} \spmqty{\delta_{r_1}\spmqty{-2 & 0 \\ -2 & 1} & \delta_{n_1}\spmqty{-\sfrac{1}{5} & \sfrac{1}{10} \\ 0 & -\sfrac{1}{10}} \\ \vb{0} & \vb{0}} \vb{x}(t - 1) \\
		                   & {}-{} \spmqty{\delta_{r_2}I_2 & \frac{1}{10}\delta_{n_2}I_2 \\ \vb{0} & \vb{0}} \vb{x}(t - 1.9) + \spmqty{1 \\ 1 \\ 0 \\ 0} v(t),
	\end{aligned}                \\
	z(t)        & = \spmqty{1 & 1 & 0 & 0} \vb{x}(t).
\end{align*}

\emph{Example~\ref{ex:rdde-1}.} Introducing the slack variable $x_4$, which
corresponds to the control signal in the open-loop system, allows us to keep the
parameters as separate matrix elements, yielding
\begin{align*}
	\spmqty{
	1                     & 0     & 0      & 0    \\
	0                     & 1     & 0      & 0    \\
	0                     & 0     & 1      & 0    \\
	0                     & 0     & 0      & 0
	} \dot{\vb{x}}(t) ={} &
	\spmqty{
	-0.08                 & -0.03 & 0.2    & 0    \\
	0.2                   & -0.04 & -0.005 & 0    \\
	-0.06                 & -0.2  & -0.07  & 0    \\
	p_1                   & p_2   & p_3    & -1
	} \vb{x}(t)                                   \\
	                      & +
	\spmqty{
	0                     & 0     & 0      & -0.1 \\
	0                     & 0     & 0      & -0.2 \\
	0                     & 0     & 0      & 0.1  \\
	0                     & 0     & 0      & 0
	} \vb{x}(t - 5) +
	\spmqty{
	1                     & 0     & 0             \\
	0                     & 1     & 0             \\
	0                     & 0     & 1             \\
	0                     & 0     & 0
	} \vb{v}(t),                                  \\
	\vb{z}(t) ={}         &
	\spmqty{
	1                     & 0     & 0      & 0    \\
	0                     & 1     & 0      & 0    \\
	0                     & 0     & 1      & 0
	} \vb{x}(t).
\end{align*}

\emph{Example~\ref{ex:rdde-2}.} As with the previous example, we can keep the
parameters to optimize as separate matrix elements by introducing the control
signal as the slack variable $x_3$, giving
\begin{align*}
	\spmqty{
	1                 & 0           & 0                  \\
	0                 & 1           & 0                  \\
	0                 & 0           & 0
	} \dot{\vb{x}}(t) & =
	\spmqty{
	0                 & 1           & 0                  \\
	-\nu^2            & -2\delta\nu & b                  \\
	-k_p              & 0           & -1
	} \vb{x}(t) +
	\spmqty{
	0                 & 0           & 0                  \\
	0                 & 0           & 0                  \\
	k_r               & 0           & 0
	} \vb{x}(t - \tau) +
	\spmqty{
	0                                                    \\
	b                                                    \\
		0
	} v(t),                                              \\
	z(t)              & = \spmqty{1 & 0  & 0} \vb{x}(t).
\end{align*}

\emph{Example~\ref{ex:ndde-1}.} Here we add the algebraic equation
$x_2(t) = \dot{x}_1(t)$ to capture the neutral terms. The control signal and input are
added as $x_3$. This then gives us
\begin{align*}
	\spmqty{
	1                 & 0           & 0                  \\
	1                 & 0           & 0                  \\
	0                 & 0           & 0
	} \dot{\vb{x}}(t) & =
	\spmqty{
	-1                & 0           & 1                  \\
	0                 & 1           & 0                  \\
	0                 & 0           & -1
	} \vb{x}(t) +
	\spmqty{
	1                 & 0           & 0                  \\
	0                 & 0           & 0                  \\
	p_2               & p_1         & 0
	} \vb{x}(t - 1) +
	\spmqty{
	0                                                    \\
	0                                                    \\
		1
	} v(t),                                              \\
	z(t)              & = \spmqty{1 & 0  & 0} \vb{x}(t).
\end{align*}

\emph{Example~\ref{ex:ndde-2}.} We get the delayed derivative and acceleration
from the algebraic equations $x_2(t) = \dot{x}_1(t)$ and
$x_3(t) = \dot{x}_2(t) = \ddot{x}_1(t)$. The slack variables
$x_4(t) = \ddot{x}_1(t - \tau_1) + v_1(t)$ and
$x_5(t) = \dot{x}_1(t - \tau_2) + v_2(t)$ are the measurements. Together this
gives
\begin{align*}
	\spmqty{
	0                     & 1         & 0 & 0   & 0                   \\
	1                     & 0         & 0 & 0   & 0                   \\
	0                     & 1         & 0 & 0   & 0                   \\
	0                     & 0         & 0 & 0   & 0                   \\
	0                     & 0         & 0 & 0   & 0
	} \dot{\vb{x}}(t) ={} &
	\spmqty{
	-1                    & -2\xi     & 0 & p_1 & p_2                 \\
	0                     & 1         & 0 & 0   & 0                   \\
	0                     & 0         & 1 & 0   & 0                   \\
	0                     & 0         & 0 & -1  & 0                   \\
	0                     & 0         & 0 & 0   & -1
	} \vb{x}(t) +
	\spmqty{
	0                     & 0         & 0 & 0   & 0                   \\
	0                     & 0         & 0 & 0   & 0                   \\
	0                     & 0         & 0 & 0   & 0                   \\
	0                     & 0         & 1 & 0   & 0                   \\
	0                     & 0         & 0 & 0   & 0
	} \vb{x}(t - \tau_1)                                              \\
	                      & +
	\spmqty{
	0                     & 0         & 0 & 0   & 0                   \\
	0                     & 0         & 0 & 0   & 0                   \\
	0                     & 0         & 0 & 0   & 0                   \\
	0                     & 0         & 0 & 0   & 0                   \\
	0                     & 1         & 0 & 0   & 0
	} \vb{x}(t - \tau_2) +
	\spmqty{
	0                     & 0                                         \\
	0                     & 0                                         \\
	0                     & 0                                         \\
	1                     & 0                                         \\
	0                     & 1
	} \vb{v}(t),                                                      \\
	z(t) ={}              & \spmqty{1 & 0 & 0   & 0   & 0} \vb{x}(t).
\end{align*}

\AtNextBibliography{\small}
\sloppy
\printbibliography

\end{document}